\documentclass{my-aims}

\usepackage{amsmath}
\usepackage{paralist}
\usepackage{graphics} 
\usepackage{epsfig} 
\usepackage{graphicx}  
\usepackage{epstopdf}
\usepackage[colorlinks=true]{hyperref}

\usepackage{subfig}

\hypersetup{urlcolor=blue, citecolor=red}

\def\ppages{X--XX}

\newtheorem{theorem}{Theorem}[section]

\newtheorem{lemma}[theorem]{Lemma}

\theoremstyle{definition}
\newtheorem{definition}[theorem]{Definition}
\newtheorem{remark}{Remark}

\newtheorem{example}{Example}

\title[A necessary condition of Pontryagin type]{A necessary condition 
of Pontryagin type for fuzzy fractional optimal control problems}

\author[O. S. Fard, J. Soolaki and D. F. M. Torres]{}



\email{osfard@du.ac.ir}
\email{javad.soolaki@gmail.com}
\email{delfim@ua.pt}

\thanks{This work is part of second author's PhD project.
It was partially supported by Damghan University, Iran;
and CIDMA--FCT, Portugal, under project UID/MAT/04106/2013.}

\thanks{$^*$Corresponding author: Delfim F. M. Torres (delfim@ua.pt)}

\begin{document}

\maketitle

\centerline{\scshape Omid S. Fard and Javad Soolaki}
\medskip
{\footnotesize
 \centerline{Department of Applied  Mathematics}
 \centerline{School of  Mathematics and Computer Science}
 \centerline{Damghan University, Damghan, Iran}} 
 
\medskip

\centerline{\scshape Delfim F. M. Torres$^*$}
\medskip
{\footnotesize
 \centerline{Center for Research \& Development in Mathematics and Applications (CIDMA)}
 \centerline{Department of Mathematics, University of Aveiro}
 \centerline{3810--193 Aveiro, Portugal}}

\bigskip


\begin{abstract}
We prove necessary optimality conditions of Pontryagin type
for a class of fuzzy fractional optimal control problems
with the fuzzy fractional derivative described in the Caputo 
sense. The new results are illustrated by computing the extremals
of three fuzzy optimal control systems, which improve recent 
results of Najariyan and Farahi.
\end{abstract}


\section{Introduction}

Optimal control problems are usually solved with the help of the famous
Pontryagin Maximum Principle (PMP), which provides a generalization
of the classical Euler--Lagrange and Weierstrass necessary optimality
conditions of the calculus of variations and is one of the central
results of the mathematics of the XX century \cite{Pinch,Pontryagin}.
On the other hand, fractional (noninteger order) derivatives play 
an increasing role in mathematics, physics and engineering
\cite{Hilfer,Kilbas,Miller,Podlubny}. The two subjects have recently
been put together and a theory of the calculus of variations and optimal control
that deals with more general systems containing noninteger order derivatives
is now available: see the books \cite{book:CMFCV,book:IFCV,book:AMFCV}.
In particular, the fractional Hamiltonian perspective is a very active 
subject, being investigated in a series of publications: see, e.g.,
\cite{Agrawal03,Torres4,MR2433010,Baleanu3,Baleanu4,MR3124694,Tarasov,MyID:295}.

Uncertainty is inherent to most real world systems and 
fuzziness is a kind of uncertainty very common 
in real word problems \cite{MR2014505}. 
In recent years, the notion of fuzzy set has been widely 
spread to various research areas, such as linear programming,
optimization, differential equations and even fractional differential equations
\cite{Salahshour}. Thus, the study of a fuzzy optimal control theory 
forms a suitable setting for the mathematical modelling of real world problems 
in which uncertainties or vagueness pervade \cite{MR3442696}.
In the past few decades, the interest in the field  of fuzzy optimal control 
has increased and fuzzy optimal control problems have attracted a great deal 
of attention. A large number of existing schemes of fuzzy optimal control 
for nonlinear systems are proposed based on the framework of the 
Takagi--Sugeno (T-S) fuzzy model originated from fuzzy identification \cite{Takagi}.
Moreover, for most of the T-S modelled nonlinear systems, fuzzy control design 
is carried out by the aid of the parallel distributed compensation (PDC) 
approach \cite{Yang}. However, it is still possible to enumerate all works 
that establish necessary optimality conditions for the fuzzy calculus 
of variations or fuzzy optimal control: see
\cite{fard2,Fard,Farhadinia,Farhadinia1,farahi,farahi1,Fard1,soolaki}. 

In \cite{farahi,farahi1}, Najariyan and Farahi obtain necessary optimality 
conditions of Pontryagin type for a very special case of fuzzy optimal 
control problems, using $\alpha$-cuts and presentation of numbers 
in a more compact form by moving to the field of complex numbers. 
The authors of \cite{farahi,farahi1} study the following fuzzy 
optimal control problem subject to a time-invariant linear control system:
\begin{equation*}
\begin{gathered}
\tilde{\mathcal{J}}(\tilde{u})
=\int_{a}^{b}\tilde{L}\left(\tilde{u}(t), t\right) dt \longrightarrow \min,\\
\dot{\tilde{x}}(t)=\tilde{A}\odot\tilde{x}(t)+\tilde{C}\odot\tilde{u}(t),\\
\tilde{x}(a)=\tilde{x}_a,\quad \tilde{x}(b)=\tilde{x}_b.
\end{gathered}
\end{equation*}
In \cite{Farhadinia1}, Farhadinia applies the fuzzy variational approach  
of \cite{Farhadinia} to fuzzy optimal control problems and derives 
necessary optimality conditions for fuzzy optimal control problems 
that depend on the Buckley and Feuring derivative (a Hukuhara derivative) 
\cite{Buckley}. In \cite{fard2,Fard}, the generalized Hukuhara derivative 
is used for a fuzzy-number-valued function, leading to solutions with 
decreasing length on their supports. Salahshour et al. \cite{Salahshour} 
and Mazandarani and Kamyad \cite{Mazandarani} proposed, respectively, 
the concepts of Riemann--Liouville and Caputo fuzzy fractional differentiability, 
based on the Hukuhara difference, which strongly generalizes 
fuzzy differentiability. 
In \cite{Allahviranloo2,Allahviranloo1},
the generalized Hukuhara fractional Riemann--Liouville 
and Caputo concepts for fuzzy-valued functions are further investigated. 
For a Hukuhara approach valid on arbitrary nonempty closed sets
of the real numbers (time scales) see \cite{MyID:341}.
In \cite{fard2}, Fard and Salehi investigate fuzzy fractional Euler--Lagrange 
equations for fuzzy fractional variational problems defined via generalized 
fuzzy fractional Caputo type derivatives. In \cite{soolaki},  
Soolaki et al. present necessary optimality conditions of Euler--Lagrange type
for variational problems with natural boundary conditions and problems 
with holonomic constraints, where the fuzzy fractional derivative 
is described in a combined sense. Here, using the PMP and a novel 
form of the Hamiltonian approach, we achieve fuzzy solutions (state and control) 
by solving an appropriate system of differential equations. The proposed method 
is not limited to just optimal fuzzy linear time-invariant controlled systems, 
which were previously studied in \cite{farahi,farahi1} for integer-order problems. 
Since the Buckley and Feuring concept of differentiability \cite{Buckley} 
or even the Hukuhara notion of differentiability are not able 
to guarantee that the obtained solutions are fuzzy functions, 
in the present work we focus on the generalized Hukuhara differentiation. 
If the order of the derivatives appearing in the formulation of our problems 
approach integer values, then one obtains via our results the extremals 
of fuzzy optimal control problems investigated 
in \cite{Farhadinia1,farahi,farahi1}. 

The paper is organized as follows. Section~\ref{1} introduces necessary notations
on fuzzy numbers and differentiability and integrability of fuzzy mappings.
The notion of Caputo generalized Hukuhara fuzzy fractional derivative is recalled 
in Section~\ref{2}. In Section~\ref{3} we establish our main result,
Theorem~\ref{the001}, that provides Pontryagin conditions 
for fuzzy fractional optimal control problems. In Section~\ref{4} 
we consider three problems, illustrating the proposed method. 
In particular, it is shown that the candidates to minimizers
given in \cite[Example~4.2]{farahi} and \cite[Example~3]{farahi2} 
are not solutions to the considered problems. We end with Section~\ref{5} 
of conclusions and future work.


\section{Preliminaries}
\label{1}

Let us denote by $\mathbb{R}_{f}$ the class of fuzzy numbers, i.e., normal, convex,
upper semicontinuous and compactly supported fuzzy subsets of the real numbers.
For $0 < r \leq 1$, let $[\tilde{u}]^{r}=\{{x} \in \mathbb{R};\tilde{u}(x)\geq r\}$
and $[\tilde{u}]^{0}=\overline{\left\{{x} \in \mathbb{R};\tilde{u}(x)\geq 0\right\}}$.
Then, it is well known that $[\tilde{u}]^{r}$ is a bounded closed interval
for any $r\in[0,1]$.

\begin{lemma}[See Theorem~1.1 of \cite{MR0825618} and Lemma~2.1 of \cite{MR2609258}]
\label{lem1}
If $\underline{a}^r:[0,1] \rightarrow \mathbb{R}$ and
$\overline{a}^r:[0,1] \rightarrow \mathbb{R}$
satisfy the conditions
\begin{enumerate}
\item[(i)]  $\underline{a}^r:[0,1] \rightarrow \mathbb{R} $ is a bounded nondecreasing function,
\item[(ii)]$\overline{a}^r:[0,1] \rightarrow \mathbb{R} $ is a bounded nonincreasing function,
\item[(iii)]  $\underline{a}^1\leq \overline{a}^1$,
\item[(iv)] for $0<k\leq 1$, $\lim_{r\rightarrow k^{-}}\underline{a}^r =\underline{a}^k$
and $\lim_{r\rightarrow k^{-}}\overline{a}^r =\overline{a}^k$,
\item[(v)] $\lim_{r\rightarrow 0^{+}}\underline{a}^r =\underline{a}^0$
and $\lim_{r\rightarrow 0^{+}}\overline{a}^r =\overline{a}^0$,
\end{enumerate}
then $\tilde{a}:\mathbb{R}\rightarrow [0,1]$, characterized by
$\tilde{a}(t)=\sup\{r|\underline{a}^r\leq t \leq \overline{a}^r\}$,
is a fuzzy number with $[\tilde{a}]^{r}=[\underline{a} ^r,\overline{a}^r]$.
The converse is also true: if $\tilde{a}(t)=\sup\{r|\underline{a}^r\leq t\leq  \overline{a}^r\}$ is
a fuzzy number with parametrization given by $ [\tilde{a}]^{r}=[\underline{a} ^r,\overline{a}^r]$,
then functions $\underline{a}^r$ and $\overline{a}^r$ satisfy conditions (i)--(v).
\end{lemma}

For $\tilde{u}, \tilde{v}\in\mathbb{R}_{f}$ and $\lambda \in\mathbb{R}$,
the sum $\tilde{u}+\tilde{v}$ and the product $\lambda \cdot \tilde{u}$
are defined by
$[\tilde{u}+\tilde{v}]^{r}=[\tilde{u}]^{r}+[\tilde{v}]^{r}$
and $[\lambda \cdotp \tilde{u}]^r=\lambda[\tilde{u}]^r$ for all $r\in[0,1]$,
where $[\tilde{u}]^{r}+[\tilde{v}]^{r}$ means the usual addition of two intervals
(subsets) of $\mathbb{R}$ and $\lambda[\tilde{u}]^r$ means the usual product
between a scalar and a subset of $\mathbb{R}$. The product $\tilde{u} \odot \tilde{v}$
of fuzzy numbers $\tilde{u}$ and $\tilde{v}$, is defined by
\begin{equation*}
[\tilde{u}\odot\tilde{v}]^r=[\min\{ \underline{u}^r \underline{v}^r,
\underline{u}^r \overline{v}^r, \overline{u}^r \underline{v}^r,
\overline{u}^r \overline{v}^r \},
\max \{\underline{u}^r \underline{v}^r, \underline{u}^r \overline{v}^r,
\overline{u}^r \underline{v}^r, \overline{u}^r \overline{v}^r \}].
\end{equation*}
The metric structure is given by the Hausdorff distance
$D:\mathbb{R}_{f} \times \mathbb{R}_{f} \rightarrow \mathbb{R}_{+}\cup\{0\}$,
$$
D(\tilde{u},\tilde{v})=\sup_{r\in[0,1]}
\max\{|\underline{u}^r-\underline{v}^r|,|\overline{u}^r-\overline{v}^r|\}.
$$
We say that the fuzzy number $\tilde{u}$ is triangular if
$\underline{u}^1=\overline{u}^1$, $\underline{u}^r
=\underline{u}^1-(1-r)(\underline{u}^1-\underline{u}^0)$
and $\overline{u}^r=\underline{u}^1-(1-r)(\overline{u}^0-\underline{u}^1)$.
The triangular fuzzy number $u$ is generally denoted by
$\tilde{u}=<\underline{u}^0,\underline{u}^1,\overline{u}^0>$.
We define the fuzzy zero $\tilde{0}_{x}$ as
$$
\tilde{0}_{x}=
\begin{cases}
1 & \text{ if } x=0,\\
0 & \text{ if } x\neq 0.
\end{cases}
$$

\begin{definition}[See \cite{Farhadinia}]
We say that $\tilde{f}:[a,b]\rightarrow \mathbb{R}_{f}$ is continuous
at $t\in[a,b]$, if both $\underline{f}^r(t)$ and $\overline{f}^r(t)$
are continuous functions of $t\in[a,b]$ for all $r \in [0,1]$.
\end{definition}

\begin{definition}[See \cite{Bede5}]
The generalized Hukuhara difference of two fuzzy numbers
$\tilde{x},\tilde{y}\in\mathbb{R}_{f}$ ($gH$-difference for short) is defined as follows:
\begin{equation*}
\tilde{x}\ominus_{gH}\tilde{y}=\tilde{z}
\Leftrightarrow
\tilde{x}=\tilde{y}+\tilde{z}
\text{ or }
\tilde{y}=\tilde{x}+(-1)\tilde{z}.
\end{equation*}
\end{definition}

If $\tilde{z}= \tilde{x}\ominus_{gH}\tilde{y}$ exists as a fuzzy number, then
its level cuts $[\underline{z}^r, \overline{z}^r]$ are obtained by
\begin{equation*}
\underline{z}^r=\min\{ \underline{x}^r - \underline{y}^r,
\overline{x}^r - \overline{y}^r\},
\quad \overline{z}^r
=\max\{ \underline{x}^r - \underline{y}^r,
\overline{x}^r - \overline{y}^r\}
\end{equation*}
for all $r\in[0,1]$.

\begin{definition}[See \cite{Hoa}]
\label{def02}
Let $t\in(a,b)$ and $h$ be such that $t+h\in(a,b)$.
The generalized Hukuhara derivative of a fuzzy-valued function
$\tilde{x}:(a,b) \rightarrow \mathbb{R}_{f}$ at $t$ is defined by
\begin{equation}
\label{G2}
\mathcal{D}_{gH}\tilde{x}(t)=\lim_{{h\to 0}}\frac{\tilde{x}(t+h)\ominus_{gH} \tilde{x}(t)}{h}.
\end{equation}
If $\mathcal{D}_{gH}\tilde{x}(t) \in \mathbb{R}_{f}$ satisfying \eqref{G2} exists,
then we say that $\tilde{x}$ is generalized Hukuhara differentiable
($gH$-differentiable for short) at $t$. Also, we say that $\tilde{x}$ is
$[(1)-gH]$-differentiable at $t$ (denoted by $\mathcal{D}_{1,gH} \tilde{x}$)
if $[\mathcal{D}_{gH}\tilde{x}(t)]^r = [\dot{\underline{x}}^r(t), \dot{\overline{x}}^r(t)]$,
and that $\tilde{x}$ is $[(2)-gH]$-differentiable at $t$
(denoted by $\mathcal{D}_{2,gH} \tilde{x}$) if
$$
[\mathcal{D}_{gH}\tilde{x}(t)]^r
= [\dot{\overline{x}}^r(t), \dot{\underline{x}}^r(t)],
\quad r\in[0,1].
$$
\end{definition}

If the fuzzy function $\tilde{f}(t)$ is continuous in the metric $D$,
then its definite integral exists. Furthermore,
\begin{equation*}
\left( \underline{\int_{a}^{b} \tilde{f}(t)dt} \right)^r
=\int_{a}^{b} \underline{f}^r (t)dt,
\quad \left( \overline{\int_{a}^{b} \tilde{f}(t)dt} \right)^r
=\int_{a}^{b} \overline{f}^r (t)dt.
\end{equation*}

\begin{definition}[See \cite{Farhadinia}]
\label{def1}
Let $\tilde{a},\tilde{b}\in\mathbb{R}_{f}$. We write $\tilde{a}\preceq\tilde{b}$,
if $\underline{a}^r \leq \underline{b}^r $ and $\overline{a}^r \leq \overline{b}^r$
for all $r \in [0,1]$. We also write $\tilde{a}\prec\tilde{b}$,
if $\tilde{a}\preceq\tilde{b}$ and there exists an $r'\in [0,1]$
so that $\underline{a}^{r'}  < \underline{b}^{r'}$ and
$\overline{a}^{r'}  < \overline{b}^{r'}$.  Moreover,
$\tilde{a}\approx\tilde{b}$ if $\tilde{a}\preceq\tilde{b}$
and $\tilde{a}\succeq\tilde{b}$, that is,
$[\tilde{a}]^r=[\tilde{b}]^r$ for all $r \in[0,1]$.
\end{definition}

We say that $\tilde{a},\tilde{b}\in\mathbb{R}_{f}$
are comparable if either $\tilde{a}\preceq\tilde{b}$
or $\tilde{a}\succeq\tilde{b}$; and noncomparable otherwise.


\section{The fuzzy fractional calculus}
\label{2}

The Riemann--Liouville fractional derivative has one disadvantage 
when modelling real world phenomena: the fractional derivative
of a constant is not zero. To eliminate this problem, one often 
considers fractional derivatives in the sense of Caputo. For this reason,
in our work we restrict ourselves to problems defined by generalized 
Hukuhara fractional Caputo derivatives. Analogous results are, however, 
easily obtained for generalized Hukuhara fractional derivatives 
in the Riemann--Liouville sense.

The fuzzy $gH$-fractional Caputo derivative of a fuzzy valued function 
was introduced in \cite{Allahviranloo2}. Following \cite{Allahviranloo2}, 
we denote the space of all continuous fuzzy valued functions on 
$[a,b] \subset \mathbb{R}$ by $C^F[a,b];$ the class of fuzzy functions
with continuous first derivatives on $[a,b]\subset \mathbb{R}$ by $C^{F1}[a,b]$;
and the space of all Lebesgue integrable fuzzy valued functions on the bounded
interval $[a,b]$ by $L^F[a,b]$.

\begin{definition}[See \cite{Arshad}]
Let $\tilde{f}(x) \in C^F[a,b]\cap L^F[a,b]$ be a fuzzy valued function and $\alpha >0$.
Then the Riemann--Liouville fractional integral of order $\alpha$ is  defined by
\begin{equation*}
{_aI_x^{\alpha}} \tilde{f}(x)
=\frac{1}{\Gamma(\alpha)}
\int_{a}^{x} \tilde{f}(t) (x-t)^{\alpha-1} dt,
\end{equation*}
where $\Gamma(\alpha)$ is the Gamma function and $x > a$.
\end{definition}

\begin{definition}[See \cite{Arshad}]
Let $\tilde{f}(x) \in C^F[a,b]\cap L^F[a,b]$ be a fuzzy valued function.
The fuzzy (left) Riemann--Liouville integral of  $\tilde{f}(x)$, based
on its $r$-level representation, can be expressed as follows:
\begin{equation*}
[{_aI_x^{\alpha}} \tilde{f}(x)]^r
=[{_aI_x^{\alpha}} \underline{f}^r(x),
{_aI_x^{\alpha}} \overline{f}^r(x)],
\quad 0\leq r \leq 1,
\end{equation*}
where
\begin{equation*}
\begin{split}
{_aI_x^{\alpha}} \underline{f}^r(x)
&=\frac{1}{\Gamma(\alpha)}
\int_{a}^{x} \underline{f}^r(t)(x-t)^{\alpha-1}dt,\\
{_aI_x^{\alpha}}\overline{f}^r(x)
&=\frac{1}{\Gamma(\alpha)}
\int_{a}^{x} \overline{f}^r(t)(x-t)^{\alpha-1}dt.
\end{split}
\end{equation*}
\end{definition}

Following \cite{Allahviranloo2,Hoa}, we now recall the definition
of Caputo-type fuzzy fractional derivative under the gH-difference.
The definition is similar to the concept of Caputo derivative
in the crisp case \cite{Podlubny} and gives a direct extension of
gH-differentiability to the fractional context \cite{Bede5}.

\begin{definition}[See \cite{Hoa}]
\label{G5}
Let $\tilde{x}(t) \in C^F[a,b]\cap L^F[a,b]$. The fuzzy $gH$-fractional Caputo derivative
of the fuzzy-valued function $\tilde{x}$ ($[gH]_\alpha^C$-differentiability for short)
is defined by
\begin{equation*}
^{gH-C}_{\qquad a}\mathcal{D}_t^\alpha \tilde{x}(t)
=\frac{1}{\Gamma(m-\alpha)}\int_{a}^{t}
(t-s)^{m-\alpha-1}(\mathcal{D}_{gH}^{(m)} \tilde{x})(s)ds,
\end{equation*}
where $m-1<\alpha < m$, $t>a$. If $\alpha \in(0,1)$, then
\begin{equation*}
^{gH-C}_{\qquad a}\mathcal{D}_t^\alpha \tilde{x}(t)
=\frac{1}{\Gamma(1-\alpha)}\int_{a}^{t} (t-s)^{-\alpha}(\mathcal{D}_{gH} \tilde{x})(s)ds.
\end{equation*}
\end{definition}

\begin{theorem}[See \cite{Allahviranloo2}]
\label{g2}
Let $\tilde{x}(t) \in C^F[a,b]\cap L^F[a,b]$ and
$[\tilde{x}(t)]^r=[\underline{x}^r(t),\overline{x}^r(t)]$
for $r \in [0,1]$, $t \in (a,b)$ and $\alpha\in (0,1)$.
The function $\tilde{x}(t)$ is $[gH]_{\alpha}^{C}$-differentiable if and only if
both $\underline{x}^r(t)$ and $\overline{x}^r(t)$ are Caputo fractional
differentiable functions. Furthermore,
\begin{equation*}
\left[^{gH-C}_{\qquad a}\mathcal{D}_t^\alpha \tilde{x}(t)\right]^r
=\Bigl[\min\left\{^C_aD_t^{\alpha}\underline{x}^r(t),
{^{C}_aD_t^{\alpha}}\overline{x}^r(t)\right\},
\max\left\{^{C}_aD_t^{\alpha}\underline{x}^r(t), {^{C}_aD_t^{\alpha}}\overline{x}^r(t)\right\}\Bigr],
\end{equation*}
where
\begin{gather*}
^{C}_aD_t^{\alpha}\underline{x}^r(t)
=\frac{1}{\Gamma(1-\alpha)}\int_{a}^{t}
(t-s)^{-\alpha} \frac{d}{ds}\underline{x}^r(s)ds,\\
^{C}_aD_t^{\alpha}\overline{x}^r(t)
=\frac{1}{\Gamma(1-\alpha)}
\int_{a}^{t} (t-s)^{-\alpha} \frac{d}{ds}\overline{x}^r(s)ds.
\end{gather*}
\end{theorem}

\begin{definition}[See \cite{Allahviranloo2}]
Let $\alpha\in[0,1]$ and $\tilde{x}:[a,b]\rightarrow \mathbb{R}_f$ be
$[gH]_{\alpha}^{C}$-differentiable at $t \in[a,b]$. We say that
$\tilde{x}$ is $[(1)-gH]_{\alpha}^{C}$-differentiable at $t \in[a,b]$ if
\begin{equation*}
[^{gH-C}_{\qquad a}\mathcal{D}_t^{\alpha}\tilde{x}(t)]^r
=\left[^{C}_aD_t^{\alpha}\underline{x}^r(t),
{^{C}_aD_t^{\alpha}}\overline{x}^r(t)\right],
\quad 0\leq r \leq 1, 
\end{equation*}
and that $\tilde{x}$ is
$[(2)-gH]_{\alpha}^{C}$-differentiable at $t $ if
\begin{equation*}
[^{gH-C}_{\qquad a}\mathcal{D}_t^{\alpha}\tilde{x}(t)]^r
=\left[^{C}_aD_t^{\alpha}\overline{x}^r(t),
{^{C}_aD_t^{\alpha}}\underline{x}^r(t)\right],
\quad 0\leq r \leq 1.
\end{equation*}
\end{definition}

\begin{remark}
We use the notation $^{gH-C}_{\qquad a}\mathcal{D}{_{it}^{\beta}} \tilde{x}$
when the fuzzy-valued function $\tilde{x}$ is $[(i)-gH]_{\alpha}^{C}$-differentiable
with respect to the independent variable $t$, $i \in \{1,2\}$.
\end{remark}

The definitions for the right fuzzy fractional operators
$_xI_b^\alpha$, $^{C}_tD_b^{\alpha}$ and
$^{gH-C}_{\qquad t}\mathcal{D}_b^{\alpha}$
of order $\alpha$, are completely analogous.


\section{Optimality of fuzzy fractional optimal control problems}
\label{3}

The fuzzy fractional optimal control problem in the sense of Caputo 
is introduced, without loss of generality, in Lagrange form:
\begin{equation}
\label{f1}
\begin{gathered}
\tilde{\mathcal{J}}(\tilde{x},\tilde{u})=\int_{a}^{b}\tilde{L}(\tilde{x}(t),
\tilde{u}(t), t) dt \longrightarrow \min, \\
^{gH-C}_{\qquad a}\mathcal{D}{_{it}^{\beta}} \tilde{x}(t)
=\tilde{\varphi}(\tilde{x}(t), \tilde{u}(t), t),
\quad i = 1, 2, \\
\tilde{x}(a)=\tilde{x}_a, \quad \tilde{x}(b)=\tilde{x}_b,
\end{gathered}
\end{equation}
where $\tilde{x}:[a,b] \rightarrow \mathbb{R}_F^n$ satisfies appropriate 
boundary conditions, $\underline{u}^r(t)$ and $\overline{u}^r(t)$ 
are piecewise continuous, and $\beta \in (0,1)$. The Lagrangian 
$\tilde{L}:\mathbb{R}_F^n \times \mathbb{R}_F^m \times [a,b]
\rightarrow \mathbb{R}_F$ and the velocity vector $\tilde{\varphi}
:\mathbb{R}_F^n \times \mathbb{R}_F^m \times [a,b] \rightarrow \mathbb{R}_F^n$
are assumed to be functions of class $C^{F1}$ with respect to all their
arguments and
\begin{equation*}
^{gH-C}_{\qquad a}\mathcal{D}{_{it}^{\beta}}\tilde{x}(t)
=\frac{1}{\Gamma(1-\beta)}\int_{a}^{t} 
(t-\tau)^{-\beta} (\mathcal{D}_{i,gH} \tilde{x})(\tau)d\tau,
\quad i=1,2.
\end{equation*}
We say that an admissible fuzzy curve $(\tilde{x}^*,\tilde{u}^*)$
is solution of problem \eqref{f1}, if for all admissible curve 
$(\tilde{x},\tilde{u})$ of problem \eqref{f1},
$$
\tilde{J}(\tilde{x}^*,\tilde{u}^*)\preceq\tilde{J}(\tilde{x},\tilde{u}).
$$
It follows, from the definition of partial ordering given 
in Definition~\ref{def1}, that the inequality 
$\tilde{J}(\tilde{x}^*,\tilde{u}^*)\preceq\tilde{J}(\tilde{x},\tilde{u})$
holds if and only if 
$$
\underline{J}^r(\underline{x}^{*r},\overline{x}^{*r},
\underline{u}^{*r},\overline{u}^{*r})\leq\underline{J}^r(\underline{x}^{r},
\overline{x}^{r},\underline{u}^{r},\overline{u}^{r})
$$ 
and 
$$
\overline{J}^r(\underline{x}^{*r},\overline{x}^{*r},\underline{u}^{*r},
\overline{u}^{*r})\leq\overline{J}^r(\underline{x}^{r},\overline{x}^{r},
\underline{u}^{r},\overline{u}^{r})
$$ 
for all $r\in[0,1]$, where the $r$-level set of fuzzy curves  
$\tilde{x}^*,\tilde{u}^*,\tilde{x}$ and $\tilde{u}$ are
$$
[\tilde{x}^*]^r=[\underline{x}^{*r},\overline{x}^{*r}],
\quad
[\tilde{u}^*]^r=[\underline{u}^{*r},\overline{u}^{*r}],
\quad [\tilde{x}]^r=[\underline{x}^{r},\overline{x}^{r}],
\quad [\tilde{u}]^r=[\underline{u}^{r},\overline{u}^{r}],
$$ 
respectively.

\begin{remark}
\label{remark1}
Choosing $\beta=1$, problem \eqref{f1} is reduced to the
fuzzy optimal control problem
\begin{equation*}
\begin{gathered}
\tilde{\mathcal{J}}(\tilde{x}(\cdot),\tilde{u}(\cdot))
=\int_{a}^{b}\tilde{L}\left(\tilde{x}(t),
\tilde{u}(t), t\right) dt \longrightarrow \min,\\
\mathcal{D}_{i,gH} \tilde{x}(t)
=\tilde{\varphi}\left(\tilde{x}(t), \tilde{u}(t), t\right)\\
\tilde{x}(a)=\tilde{x}_a,
\quad \tilde{x}(b)=\tilde{x}_b,
\end{gathered}
\end{equation*}
which is studied in \cite{Farhadinia1}.
\end{remark}

\begin{remark}
\label{remark2}
The fuzzy fractional problem of the calculus of variations 
in the sense of Caputo,
\begin{equation*}
\begin{gathered}
\tilde{\mathcal{J}}(\tilde{x}(\cdot))
=\int_{a}^{b}\tilde{L}\left(\tilde{x}(t),
{^{gH-C}_{\qquad a}D_{it}^\beta} \tilde{x}(t), t\right)dt \longrightarrow \min,\\
\tilde{x}(a)=\tilde{x}_a,\quad \tilde{x}(b)=\tilde{x}_b,
\end{gathered}
\end{equation*}
was first introduced in \cite{fard2} and is a particular 
case of our problem \eqref{f1}: one just need to choose 
$\tilde{\varphi}(\tilde{x} ,\tilde{u},t)=\tilde{u}$.
\end{remark}
 
\begin{theorem}[Pontryagin Maximum Principle for problem \eqref{f1}]
\label{the001}
Let control $\tilde{u}^*$ have the lower and upper bounds
$\underline{u}^{*r}$ and $\overline{u}^{*r}$, and $\tilde{x}^*$ be the corresponding
state with lower and upper bounds $\underline{x}^{*r}$ and $\overline{x}^{*r}$.
If $(\tilde{x}^*,\tilde{u}^*)$ is solution to \eqref{f1}, then there
exist costate functions $p_1$ and $p_2$ such that the
quadruple $\left(\underline{x}^{*r}, \overline{x}^{*r},
\underline{u}^{*r}, \overline{u}^{*r}\right)$ satisfies
\begin{itemize}
\item the Hamiltonian adjoint system
\begin{equation*}
\displaystyle {_tD{_b^{\beta}}}p_1^r(t)
=\frac{\partial\mathcal{H}}{\partial \underline{x}^{*r}},
\quad \displaystyle {_tD{_b^{\beta}}}p_2^{r}(t) 
=\frac{\partial\mathcal{H}}{\partial \overline{x}^{*r}},
\end{equation*}
\item and the stationary conditions
\begin{equation*}
\displaystyle \frac{\partial\mathcal{H}}{\partial \underline{u}^r}=0,
\quad \displaystyle \frac{\partial\mathcal{H}}{\partial \overline{u}^r}=0,
\end{equation*}
\end{itemize}
where the partial derivatives are evaluated at
$$
\left(\underline{x}^{*r}(t), \overline{x}^{*r}(t),
\underline{u}^{*r}(t), \overline{u}^{*r}(t), p_1(t), p_2(t), t\right)
$$
with the Hamiltonian $\mathcal{H}$ defined as follows:
if $\tilde{x}^{*}$ is [(1)-gH]-differentiable, then
\begin{multline}
\label{eq:Hamilt:01}
\mathcal{H}\left(\underline{x}^r, \overline{x}^r, \underline{u}^r, 
\overline{u}^r, p_1, p_2, t\right)
=-(\underline{L}^r(\underline{x}^r, \overline{x}^r, \underline{u}^r, \overline{u}^r, t)
+\overline{L}^r(\underline{x}^r, \overline{x}^r, \underline{u}^r, \overline{u}^r, t))\\
+p_1\cdot\underline{\varphi}^r(\underline{x}^r, \overline{x}^r, \underline{u}^r,
\overline{u}^r,t)+p_2\cdot\overline{\varphi}^r(\underline{x}^r, \overline{x}^r,
\underline{u}^r, \overline{u}^r, t);
\end{multline}
if $\tilde{x}^{*}$ is [(2)-gH]-differentiable, then
\begin{multline}
\mathcal{H}(\underline{x}^r, \overline{x}^r, \underline{u}^r, 
\overline{u}^r, p_1, p_2, t)
=-(\underline{L}^r(\underline{x}^r, \overline{x}^r, 
\underline{u}^r, \overline{u}^r,t)+\overline{L}^r(\underline{x}^r, 
\overline{x}^r, \underline{u}^r, \overline{u}^r, t))\\
+ p_1\cdot\overline{\varphi}^r(\underline{x}^r, \overline{x}^r, 
\underline{u}^r, \overline{u}^r, t)+p_2\cdot\underline{\varphi}^r(\underline{x}^r,
\overline{x}^r, \underline{u}^r, \overline{u}^r, t).
\end{multline}
\end{theorem}

\begin{proof}
Consider a variation 
$\underline{u}^{r}=\underline{u}^{*r}+\delta\underline{u}^{r}$
and a variation 
$\overline{u}^{r}=\overline{u}^{*r}+\delta\overline{u}^{r}$
of $\underline{u}^{*r}$ and $\overline{u}^{*r}$, respectively,
with corresponding state $(\underline{x}^{*r}+\delta\underline{x}^{r},
\overline{x}^{*r}+\delta\overline{x}^{r})$. The consequent change 
$\tilde{\Delta}(\tilde{\mathcal{J}})$ in $\tilde{\mathcal{J}}$ is
\begin{equation*}
\tilde{\Delta}(\tilde{\mathcal{J}})
=\int_{a}^{b}\tilde{L}(\tilde{x}^*+\delta \tilde{x},\tilde{u}^*+\delta \tilde{u},t)dt
\ominus_{gH} \int_{a}^{b}\tilde{L}(\tilde{x}^*,\tilde{u}^*,t)dt.
\end{equation*}
Denote $[\Delta\tilde{\mathcal{J}}]^{r}
=[\underline{\Delta\mathcal{J}}^r,\overline{\Delta\mathcal{J}}^r]$.
Using the gH-difference, one gets
\begin{multline*}
\underline{\Delta\mathcal{J}}^r
=\min\left\{\int_{a}^{b} \underline{L}^r [\tilde{x}+\delta\tilde{x},
\tilde{u}+\delta\tilde{u}]^r dt
-\int_{a}^{b} \underline{L}^r [\tilde{x}, \tilde{u}]^r dt,\right.\\
\left.\int_{a}^{b} \overline{L}^r [\tilde{x}
+\delta\tilde{x}, \tilde{u}+\delta\tilde{u}]^r dt
-\int_{a}^{b} \overline{L}^r [\tilde{x}, \tilde{u}]^r dt\right\},
\end{multline*}
\begin{multline*}
\overline{\Delta\mathcal{J}}^r
=\max\left\{\int_{a}^{b} \underline{L}^r [\tilde{x}+\delta\tilde{x},
\tilde{u}+\delta\tilde{u}]^r dt
-\int_{a}^{b} \underline{L}^r [\tilde{x}, \tilde{u}]^r dt,\right.\\
\left.\int_{a}^{b} \overline{L}^r [\tilde{x}+\delta\tilde{x}, 
\tilde{u}+\delta\tilde{u}]^r dt
-\int_{a}^{b} \overline{L}^r [\tilde{x}, \tilde{u}]^r dt\right\},
\end{multline*}
where 
\begin{equation*}
\begin{split}
[\tilde{x}+\delta\tilde{x}, \tilde{u}+\delta\tilde{u}]^r
&=(\underline{x}^{*r}+\delta\underline{x}^{r}, \overline{x}^{*r}
+\delta\overline{x}^{r}, \underline{u}^{*r}+\delta\underline{u}^{r},
\overline{u}^{*r}+\delta\overline{u}^{r}, t),\\ 
[\tilde{x}, \tilde{u}]^r
&=(\underline{x}^{*r}, \overline{x}^{*r}, 
\underline{u}^{*r}, \overline{u}^{*r}, t).
\end{split}
\end{equation*}
Without loss of generality, we consider
\begin{multline*}
\underline{\Delta\mathcal{J}}^r =\int_{a}^{b}
\underline{L}^r\bigl(\underline{x}^{*r}+\delta\underline{x}^{r}, \overline{x}^{*r}
+\delta\overline{x}^{r}, \underline{u}^{*r}+\delta\underline{u}^{r},
\overline{u}^{*r} +\delta\overline{u}^{r}, t\bigr) dt\\
-\int_{a}^{b} \underline{L}^r\left(\underline{x}^{*r},
\overline{x}^{*r}, \underline{u}^{*r}, \overline{u}^{*r}, t\right) dt
\end{multline*}
and
\begin{multline*}
\overline{\Delta\mathcal{J}}^r =\int_{a}^{b}
\overline{L}^r \bigl(\underline{x}^{*r}+\delta\underline{x}^{r}, \overline{x}^{*r}
+\delta\overline{x}^{r}, \underline{u}^{*r}+\delta\underline{u}^{r},
\overline{u}^{*r}+\delta\overline{u}^{r}, t\bigr) dt\\
-\int_{a}^{b} \overline{L}^r (\underline{x}^{*r},
\overline{x}^{*r}, \underline{u}^{*r}, \overline{u}^{*r}, t)dt.
\end{multline*}
If we evaluate the derivatives in the integrand along 
the optimal trajectory, then we arrive at
\begin{equation*}
\underline{\Delta\mathcal{J}}^r
=\int_{a}^{b}
\biggl[\frac{\partial \underline{L}^r}{\partial \underline{x}^r} \delta \underline{x}^r
+ \frac{\partial \underline{L}^r}{\partial \overline{x}^r} \delta \overline{x}^r
+ \frac{\partial \underline{L}^r}{\partial \underline{u}^r} \delta \underline{u}^r
+ \frac{\partial \underline{L}^r}{\partial \overline{u}^r} 
\delta  \overline{u}^r\biggr] dt + O((\delta  \underline{u}^{*r})^2) 
+ O((\delta  \overline{u}^{*r})^2)
\end{equation*}
and
\begin{equation*}
\overline{\Delta\mathcal{J}}^r
= \int_{a}^{b}
\biggl[\frac{\partial \overline{L}^r}{\partial \underline{x}^r} \delta \underline{x}^r
+ \frac{\partial \overline{L}^r}{\partial \overline{x}^r} \delta \overline{x}^r
+ \frac{\partial \overline{L}^r}{\partial \underline{u}^r} \delta \underline{u}^r
+ \frac{\partial \overline{L}^r}{\partial \overline{u}^r} \delta \overline{u}^r\biggr] dt
+ O((\delta  \underline{u}^{*r})^2) + O((\delta  \overline{u}^{*r})^2).
\end{equation*}
Since 
$\tilde{\mathcal{J}}(\tilde{x}^*,\tilde{u}^*)
\preceq\tilde{\mathcal{J}}(\tilde{x},\tilde{u})$ if and only if  
$\underline{\mathcal{J}}^r[\tilde{x}^*,
\tilde{u}^*]^r\leq\underline{\mathcal{J}}^r[\tilde{x},\tilde{u}]^r$  
and $\overline{\mathcal{J}}^r[\tilde{x}^*,
\tilde{u}^*]^r\leq\overline{\mathcal{J}}^r[\tilde{x},\tilde{u}]^r$  
for all $r\in[0,1]$, so $[\tilde{x}^*,\tilde{u}^*]^r$ is an optimal 
solution for the crisp functions $\underline{\mathcal{J}}^r$ 
and $\overline{\mathcal{J}}^r$. 
Let $\underline{\delta \mathcal{J}}^r$ and $\overline{\delta \mathcal{J}}^r$
denote the first variation. If $\underline{u}^{*r}$ and  $\overline{u}^{*r}$
are optimal, from the classical theory of optimal control,  
it is necessary that the first variation $\underline{\delta \mathcal{J}}^r$
and $\overline{\delta \mathcal{J}}^r$ are zero. Thus, on optimal trajectories, 
one has
\begin{equation*}
\underline{\delta \mathcal{J}}^r=\int_{a}^{b}
\left[\frac{\partial \underline{L}^r}{\partial \underline{x}^r} \delta \underline{x}^r
+ \frac{\partial \underline{L}^r}{\partial \overline{x}^r} \delta \overline{x}^r
+ \frac{\partial \underline{L}^r}{\partial \underline{u}^r} \delta \underline{u}^r
+ \frac{\partial \underline{L}^r}{\partial \overline{u}^r} 
\delta \overline{u}^r \right]dt =0
\end{equation*}
and
\begin{equation*}
\overline{\delta \mathcal{J}}^r=\int_{a}^{b}
\left[\frac{\partial \overline{L}^r}{\partial \underline{x}^r} \delta \underline{x}^r
+ \frac{\partial \overline{L}^r}{\partial \overline{x}^r} \delta \overline{x}^r
+ \frac{\partial \overline{L}^r}{\partial \underline{u}^r} \delta \underline{u}^r
+ \frac{\partial \overline{L}^r}{\partial \overline{u}^r} 
\delta \overline{u}^r\right] dt = 0
\end{equation*}
for all variations. Now, we simply need to introduce 
two Lagrange multipliers $p_1(t)$ and $p_2(t)$.
If $\tilde{x}$ is [(1)-gH]-differentiable, then  we consider the integrals
\begin{equation}
\label{equ08}
\underline{\phi}^r=\int_{a}^{b} p_1 \cdot
(_{a}^CD{_{t}^{\beta}} \underline{x}^r-\underline{\varphi}^r)  dt
\end{equation}
and
\begin{equation}
\label{equ09}
\overline{\phi}^r=\int_{a}^{b} p_2 \cdot (_{a}^CD{_{t}^{\beta}}
\overline{x}^r-\overline{\varphi}^r)  dt.
\end{equation}
If $\tilde{x}$ is [(2)-gH]-differentiable, then we consider the integrals
\begin{equation*}
\underline{\phi}^r=\int_{a}^{b}p_1
\cdot (_{a}^CD{_{t}^{\beta}} \underline{x}^r-\overline{\varphi}^r)  dt
\end{equation*}
and
\begin{equation*}
\overline{\phi}^r=\int_{a}^{b} p_2 \cdot (_{a}^CD{_{t}^{\beta}}
\overline{x}^r-\underline{\varphi}^r)  dt.
\end{equation*}
Let us assume [(1)-$gH$]-differentiablity of $\tilde{x}^*$.
The proof for the other case is completely similar, so it is here omitted.
We begin by computing the variation $\underline{\delta\phi}^r$ 
of functional \eqref{equ08}:
\begin{equation}
\label{equ10}
\begin{split}
\underline{\delta{\phi}}^r &= \int_{a}^{b} \delta p_1 (_{a}^CD{_{t}^{\beta}}
\underline{x}^r-\underline{\varphi}^r)\\
&\qquad + p_1 \biggl[\delta (_{a}^CD{_{t}^{\beta}} \underline{x}^r)
- \left(\frac{\partial\underline{\varphi}^r}{\partial\underline{x}^r} \delta
\underline{x}^r+\frac{\partial\underline{\varphi}^r}{\partial\overline{x}^r} \delta
\overline{x}^r+\frac{\partial\underline{\varphi}^r}{\partial\underline{u}^r} \delta
\underline{u}^r+\frac{\partial\underline{\varphi}^r}{\partial\overline{u}^r} \delta
\overline{u}^r\right)\biggr]dt\\
&= \int_{a}^{b} p_1 \biggl[\delta \left(_{a}^CD{_{t}^{\beta}} \underline{x}^r\right)
-\biggl( \frac{\partial\underline{\varphi}^r}{\partial\underline{x}^r} 
\delta \underline{x}^r
+\frac{\partial\underline{\varphi}^r}{\partial\overline{x}^r} \delta \overline{x}^r
+\frac{\partial\underline{\varphi}^r}{\partial\underline{u}^r} \delta \underline{u}^r
+\frac{\partial\underline{\varphi}^r}{\partial\overline{u}^r} 
\delta \overline{u}^r\biggr)\biggr]dt.
\end{split}
\end{equation}
Because $\tilde{x}(a)$ and $\tilde{x}(b)$ are specified, we have
$\delta\underline{x}^r(a)=\delta\underline{x}^r(b)
=\delta\overline{x}^r(a)=\delta\overline{x}^r(b)=0$.
Using fractional integration by parts \cite{book:IFCV}, 
equation \eqref{equ10} is equivalent to
\begin{equation*}
\underline{\delta{\phi}}^r=\int_{a}^{b} \left(_{t}D{_{b}^{\beta}}
p_1-p_1\frac{\partial\underline{\varphi}^r}{\partial\underline{x}^r}\right)
\delta \underline{x}^r
-\left(\frac{\partial\underline{\varphi}^r}{\partial\overline{x}^r}
\delta \overline{x}^r+\frac{\partial\underline{\varphi}^r}{\partial\underline{u}^r}
\delta \underline{u}^r+\frac{\partial\underline{\varphi}^r}{\partial\overline{u}^r}
\delta \overline{u}^r\right)p_1 dt
\end{equation*}
since $\underline{\phi}^r=0$ for all $\underline{u}^r$ and $\overline{u}^r$
and $\underline{\delta {\phi}}^r=0$. Therefore, 
the condition $\underline{\delta\mathcal{J}}^r=0$
can now be replaced by $\underline{\delta\mathcal{J}}^r+\underline{\delta {\phi}}^r=0$.
With substitutions of $\underline{\delta\mathcal{J}}^r$ 
and $\underline{\delta {\phi}}^r$, we have
\begin{equation}
\label{equ11}
\begin{split}
\int_{a}^{b} & \left(\frac{\partial\underline{L}^r}{\partial\underline{x}^r}
+_{t}D{_{b}^{\beta}} p_1-p_1\frac{\partial\underline{\varphi}^r}
{\partial\underline{x}^r}\right) \delta \underline{x}^r
+\left(\frac{\partial\underline{L}^r}{\partial\overline{x}^r}
-p_1\frac{\partial\underline{\varphi}^r}
{\partial\overline{x}^r}\right)\delta\overline{x}^r
+\left(\frac{\partial\underline{L}^r}{\partial\underline{u}^r}
-p_1\frac{\partial\underline{\varphi}^r}{\partial\underline{u}^r}\right)
\delta \underline{u}^r\\
&+\left(\frac{\partial\underline{L}^r}{\partial\overline{u}^r}
-p_1\frac{\partial\underline{\varphi}^r}{\partial\overline{u}^r}\right)
\delta \overline{u}^r dt=0.
\end{split}
\end{equation}
Now, following the scheme of obtaining Eq.~\eqref{equ11}, and adapting 
it to the case under consideration involving Eq.~\eqref{equ09}, 
the condition $\overline{\delta\mathcal{J}}^{r}=0$ can be replaced by
$\overline{\delta\mathcal{J}}^{r}+\overline{\delta\phi}^{r}=0$. So we have
\begin{equation}
\label{equ12}
\begin{split}
\int_{a}^{b} &\left(\frac{\partial\overline{L}^r}{\partial\underline{x}^r}
-p_2\frac{\partial\overline{\varphi}^r}{\partial\underline{x}^r}\right)\delta
\underline{x}^r 
+\left(\frac{\partial\overline{L}^r}{\partial\underline{x}^r}
+_{t}D{_{b}^{\beta}} p_2-p_2\frac{\partial\overline{\varphi}^r}
{\partial\overline{x}^r}\right) \delta \overline{x}^r 
+\left(\frac{\partial\overline{L}^r}{\partial\underline{u}^r}
-p_2\frac{\partial\overline{\varphi}^r}{\partial\underline{u}^r}\right)
\delta \underline{u}^r\\
&+\left(\frac{\partial\overline{L}^r}{\partial\overline{u}^r}
-p_2\frac{\partial\overline{\varphi}^r}{\partial\overline{u}^r}\right)
\delta \overline{u}^r dt=0.
\end{split}
\end{equation}
If we use the Hamiltonian function as in \eqref{eq:Hamilt:01},
then by summing Eqs.~\eqref{equ11} and \eqref{equ12} we arrive at
\begin{equation*}
\int_{a}^{b} \left({_{t}D{_{b}^{\beta}}} p_1
-\frac{\partial \mathcal{H}}{\partial\underline{x}^r}\right)\delta\underline{x}^r
+ \left({_{t}D{_{b}^{\beta}}} p_2
-\frac{\partial \mathcal{H}}{\partial\overline{x}^r}\right)\delta\overline{x}^r
-\frac{\partial \mathcal{H}}{\partial\underline{u}^r}\delta\underline{u}^r
-\frac{\partial \mathcal{H}}{\partial\overline{u}^r}\delta\overline{u}^r dt=0.
\end{equation*}
The intended necessary conditions follow.
\end{proof}

A pair $(\tilde{x}^*, \tilde{u}^*)$ satisfying Theorem~\ref{the001}
is said to be an extremal for problem \eqref{f1}.


\section{Illustrative examples}
\label{4}

In this section, we apply the necessary conditions of Pontryagin type
given by Theorem~\ref{the001} to three fuzzy optimal control problems.


\subsection{A non-autonomous fuzzy fractional optimal control problem}

We begin with a non-autonomous fuzzy fractional optimal control problem.

\begin{example}
\label{example}
Consider the following problem:
\begin{equation}
\label{104}
\begin{gathered}
\int_{1}^{2} \tilde{u}^2 (t)dt \longrightarrow \min, \\
^{gH-C}_{\qquad 1}\mathcal{D}{_{t}^{\beta}}\tilde{x} (t) 
= (2t-1)   \tilde{x}(t) \ominus_{gH} \sin(t) \tilde{u}(t),\\
\tilde{x}(1)= (0,1,2), \quad \tilde{x}(2)=(-2,-1,1).
\end{gathered}
\end{equation}
We assume that $(2t-1)\tilde{x}(t) \ominus_{gH} \sin(t) \tilde{u}(t)$ 
exists and
$$
[(2t-1)\tilde{x}(t) \ominus_{gH}\sin(t)\tilde{u}(t)]^r
=[(2t-1)\underline{x}^r-\sin(t) \underline{u}^r,(2t-1)\overline{x}^r
-\sin(t) \overline{u}^r].
$$
Using Theorem~\ref{the001}, we consider two cases
to obtain the extremals of \eqref{104}. 

(i) Suppose that $\tilde{x}$ is a $[(1)-gH]_{\beta}^{C}$-differentiable 
function. Then, the Hamiltonian is given by
$$
\mathcal{H}=-((\underline{u}^r)^2+(\overline{u}^r)^2)+p_1((2t-1)\underline{x}^r
-\sin(t)\underline{u}^r)+p_2((2t-1)\overline{x}^r-\sin(t)\overline{u}^r).
$$
The optimality conditions of Theorem~\ref{the001}, the initial conditions, 
and the control system of \eqref{104} assert that
\begin{equation}
\label{l07}
\begin{gathered}
\begin{cases}
_{t}D{_{2}^{\beta}}p_1(t)=(2t-1)p_1(t),\\
_{t}D{_{2}^{\beta}}p_2(t)=(2t-1)p_2(t),\\
{\underline{u}}^r(t)= -\frac{p_1(t) \sin(t)}{2},\\
{\overline{u}}^r(t)= -\frac{p_2(t) \sin(t)}{2},
\end{cases}
\quad
\begin{cases}
_{1}D{_{t}^{\beta}}{\underline{x}}^r(t)
= (2t-1)\underline{x}^r(t)-\sin(t)\underline{u}^r(t),\\
_{1}D{_{t}^{\beta}}{\overline{x}}^r(t)
=(2t-1)\overline{x}^r(t)-\sin(t)\overline{u}^r(t),
\end{cases}\\
\begin{cases}
\underline{x}^r(1)= r,\\
\overline{x}^r(1)=2-r,\\
\underline{x}^r(2)=-2+r,\\
\overline{x}^r(2)=1-2r.
\end{cases}
\end{gathered}
\end{equation}
Note that it is difficult to solve the above fractional equations 
to get the extremals. For $0<\beta<1$, a numerical method should 
be used \cite{book:CMFCV,MR3381791}. When $\beta$ goes to 1, 
problem \eqref{104} reduces to
\begin{equation}
\label{105}
\begin{gathered}
\int_{1}^{2} \tilde{u}^2 (t)dt \longrightarrow \min, \\
\mathcal{D}_{gH}\tilde{x} (t) = (2t-1) 
\odot \tilde{x}(t) \ominus_{gH} \sin(t) \tilde{u}(t),\\
\tilde{x}(1)= (0,1,2), \quad \tilde{x}(2)=  (-2,-1,1).
\end{gathered}
\end{equation}
The extremals for \eqref{105} are obtained from \eqref{l07} 
by considering $\beta\rightarrow 1$:
\begin{equation}
\label{l08}
\begin{gathered}
\begin{cases}
\dot{p}_1(t)=(1-2t){p}_1(t),\\
\dot{p}_2(t)=(1-2t){p}_2(t),\\
{\underline{u}}^r(t)= -\frac{{p}_1(t) \sin(t)}{2},\\
{\overline{u}}^r(t)= -\frac{{p}_2(t) \sin(t)}{2},
\end{cases}
\quad
\begin{cases}
\dot{\underline{x}}^r(t)
= (2t-1)\underline{x}^r(t)-\sin(t)\underline{u}^r(t),\\
\dot{\overline{x}}^r(t)
= (2t-1)\overline{x}^r(t)-\sin(t)\overline{u}^r(t),
\end{cases}\\
\begin{cases}
\underline{x}^r(1)= r,\\
\overline{x}^r(1)=2-r,\\
\underline{x}^r(2)=-2+r,\\
\overline{x}^r(2)=1-2r.
\end{cases}
\end{gathered}
\end{equation}
We solved \eqref{l08} numerically, with {\sc Matlab}'s built-in 
solver \texttt{bvp4c}. Figure~\ref{fig005} shows the control 
and state extremals, where the solid lines in the center 
corresponds to $r=1$, the dashed lines are the upper bounds 
and the doted lines are the lower bounds, for both fuzzy
control and state functions, correspondent to $r=0$.
\begin{figure}[tbp]
\label{03}
\centering
\subfloat[Fuzzy control extremal]{\label{fig005:a}
\includegraphics[scale=0.42]{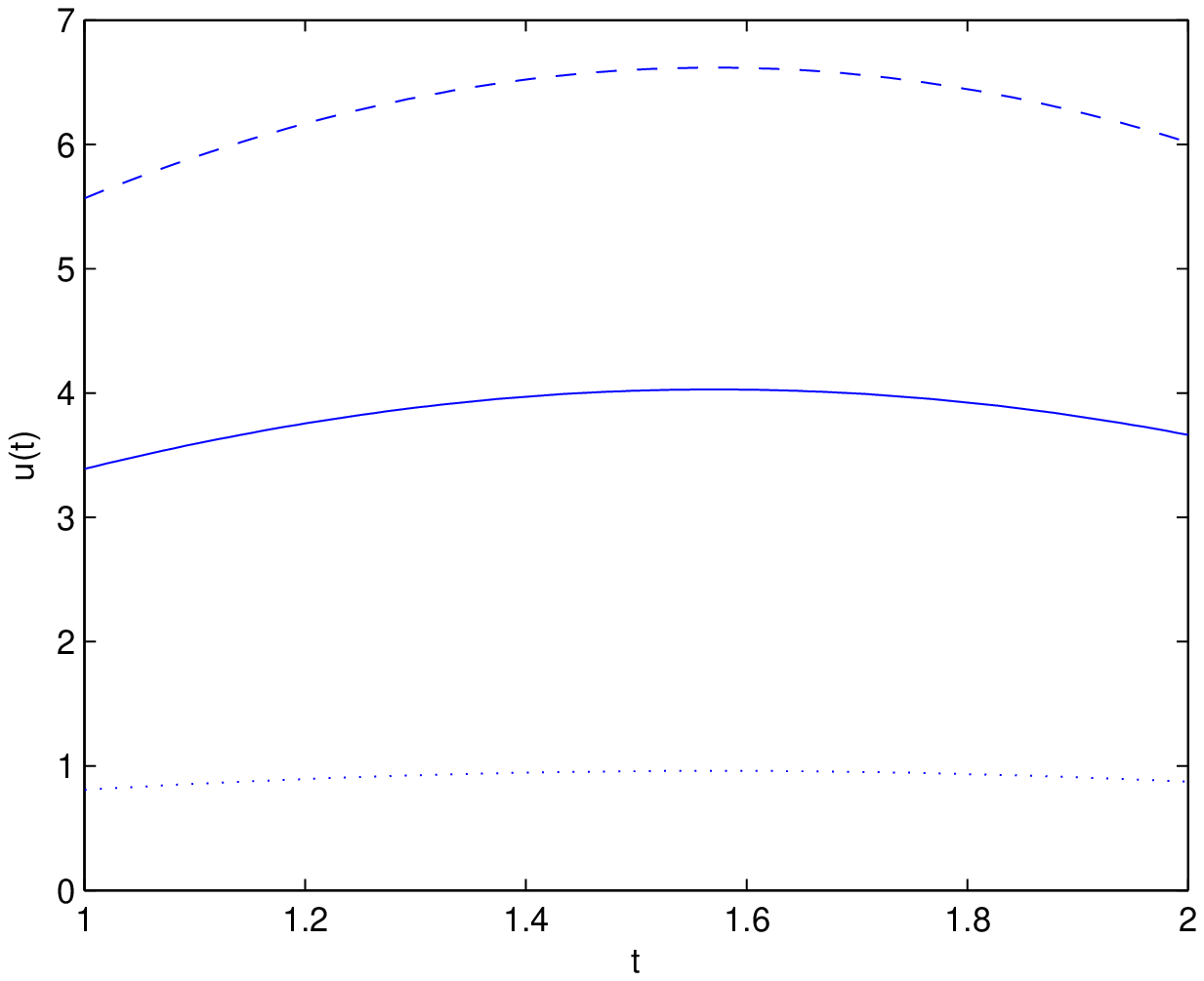}}
\subfloat[Fuzzy state extremal]{\label{fig005:b}
\includegraphics[scale=0.42]{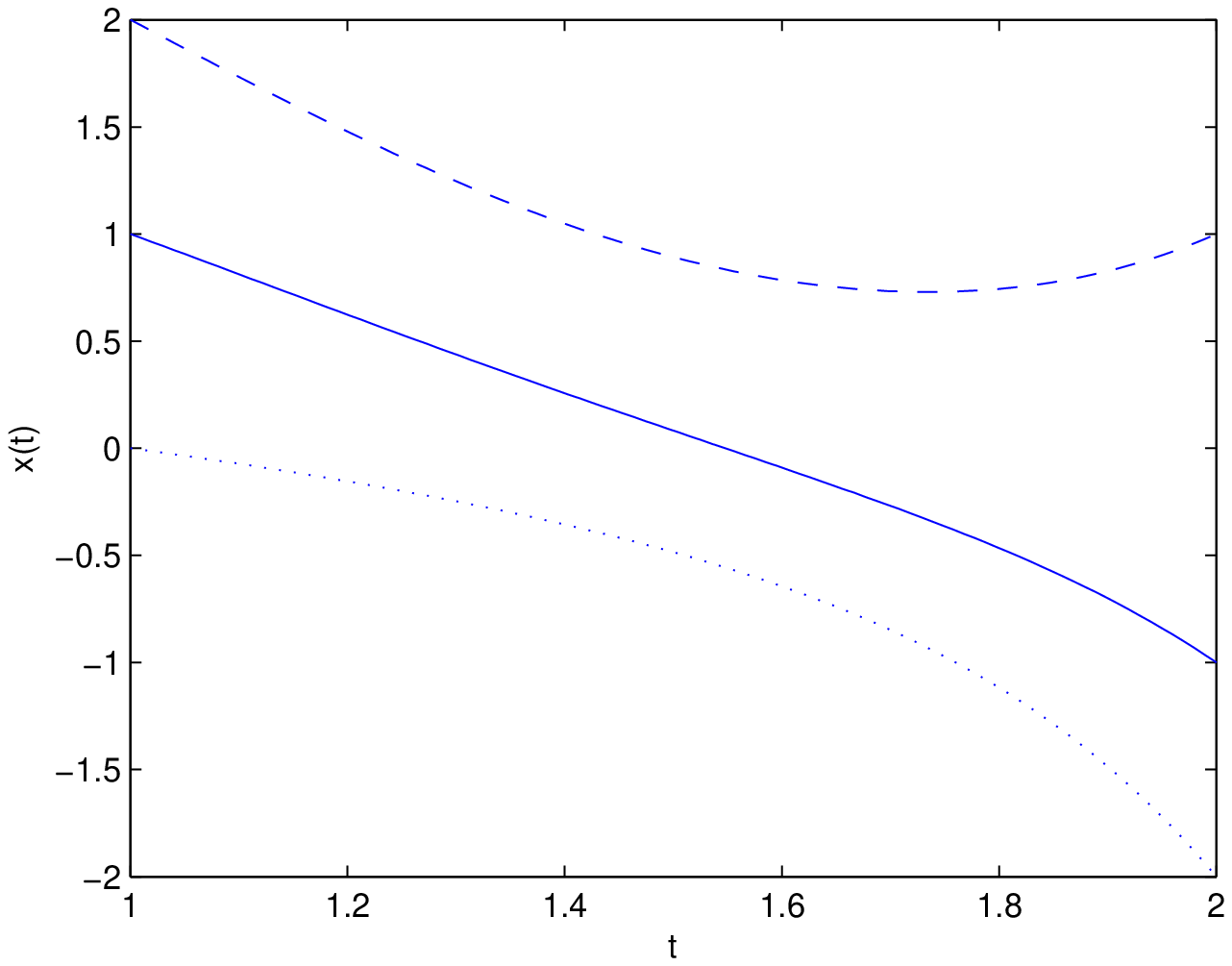}}
\caption{The fuzzy extremals for the fuzzy optimal control 
problem \eqref{105} of Example~\ref{example} under
$[(1)-gH]_{\beta}^{C}$-differentiability 
of $\tilde{x}$.}\label{fig005}
\end{figure}

(ii) Suppose now that $\tilde{x}$ is $[(2)-gH]_{\beta}^{C}$-differentiable. 
This leads to
$$
\mathcal{H}=-((\underline{u}^r)^2+(\overline{u}^r)^2)
+p_1((2t-1)\overline{x}^r-\sin t \overline{u}^r)
+p_2((2t-1)\underline{x}^r-\sin t \underline{u}^r).
$$
The optimality conditions of Theorem~\ref{the001}, 
the initial conditions and the control system assert that
\begin{equation}
\label{l007}
\begin{gathered}
\begin{cases}
_{t}D{_{2}^{\beta}}p_1(t)=(2t-1)p_2(t),\\
_{t}D{_{2}^{\beta}}p_2(t)=(2t-1)p_1(t),\\
{\underline{u}}^r(t)= -\frac{p_2(t) \sin(t)}{2},\\
{\overline{u}}^r(t)= -\frac{p_1(t) \sin(t)}{2},
\end{cases}
\quad
\begin{cases}
_{1}D{_{t}^{\beta}}{\underline{x}}^r
= (2t-1)\overline{x}^r-\sin(t)\overline{u}^r,\\
_{1}D{_{t}^{\beta}}{\overline{x}}^r
= (2t-1)\underline{x}^r-\sin(t)\underline{u}^r,\\
\end{cases}\\
\begin{cases}
\underline{x}^r(1)= r,\\
\overline{x}^r(1)=2-r,\\
\underline{x}^r(2)=-2+r,\\
\overline{x}^r(2)=1-2r.
\end{cases}
\end{gathered}
\end{equation}
Note that it is difficult to solve the above fractional equations 
to get the extremals. For $0<\beta<1$, a numerical method should be used. 
When $\beta$ goes to 1, problem \eqref{104} reduces to problem \eqref{105}.
The extremals for \eqref{105} are obtained from \eqref{l007} 
and considering $\beta\rightarrow 1$:
\begin{equation}
\label{l008}
\begin{cases}
\dot{p}_1(t)=(1-2t){p}_2(t),\\
\dot{p}_2(t)=(1-2t){p}_1(t),\\
{\underline{u}}^r(t)= -\frac{{p}_2(t) \sin(t)}{2},\\
{\overline{u}}^r(t)= -\frac{{p}_1(t) \sin(t)}{2},
\end{cases}
\ \ 
\begin{cases}
\dot{\underline{x}}^r= (2t-1)\overline{x}^r-\sin(t)\overline{u}^r,\\
\dot{\overline{x}}^r= (2t-1)\underline{x}^r-\sin(t)\underline{u}^r,
\end{cases}
\ \ 
\begin{cases}
\underline{x}^r(1)= r,\\
\overline{x}^r(1)=2-r,\\
\underline{x}^r(2)=-2+r,\\
\overline{x}^r(2)=1-2r.
\end{cases}
\end{equation}
Similarly as before, we solved \eqref{l008} with {\sc Matlab}'s 
built-in solver \texttt{bvp4c}.  Figure~\ref{fig0005} shows the graphic
of the control and  state extremals, where the solid lines at the center correspond
to $r=1$, the dashed lines are the upper bounds, 
and the doted lines are the lower bounds for fuzzy
control and state functions for $r=0$.
\begin{figure}[tbp]
\label{003}
\centering
\subfloat[Fuzzy control extremal]{\label{fig0005:a}
\includegraphics[scale=0.42]{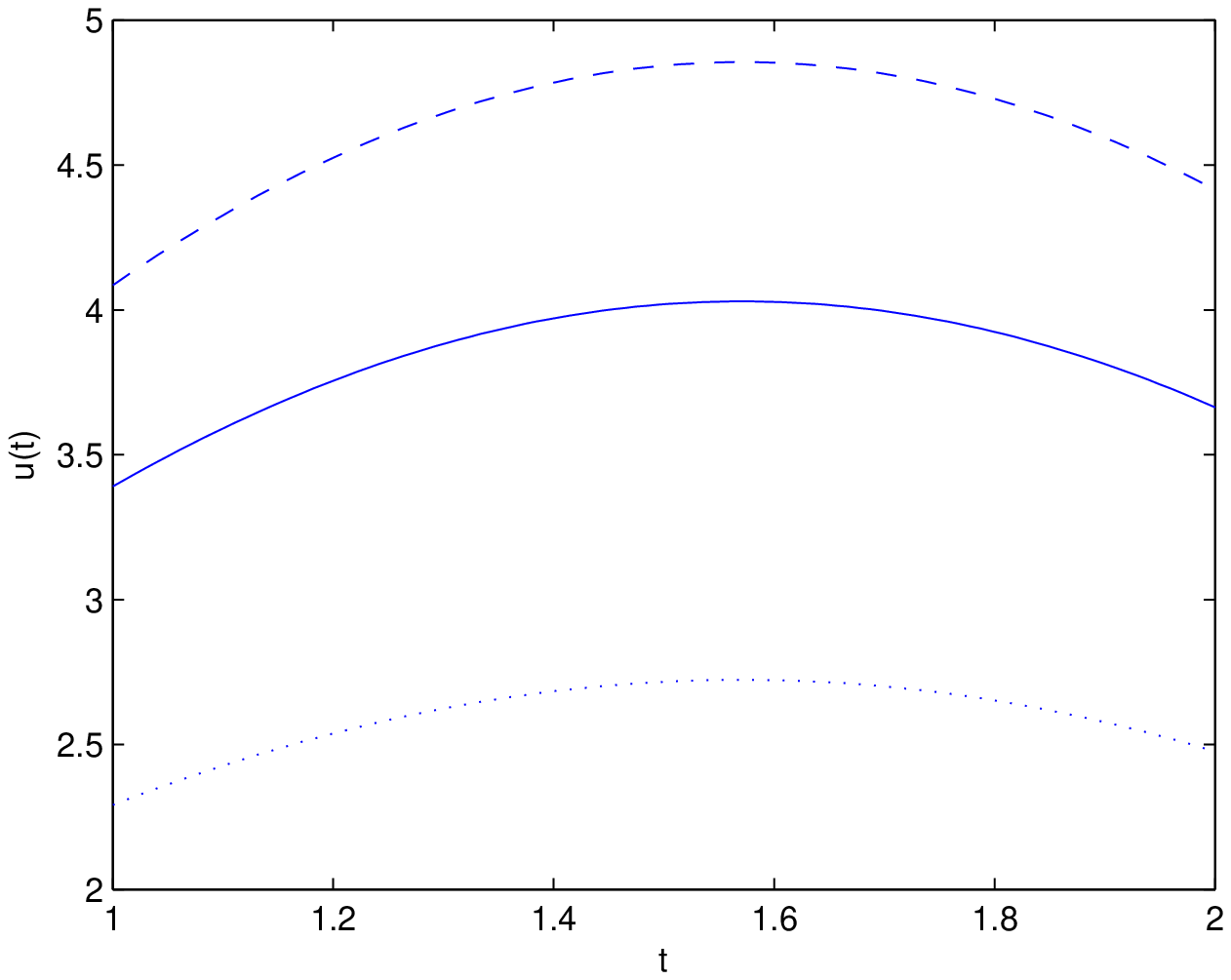}}
\subfloat[Fuzzy state extremal]{\label{fig0005:b}
\includegraphics[scale=0.42]{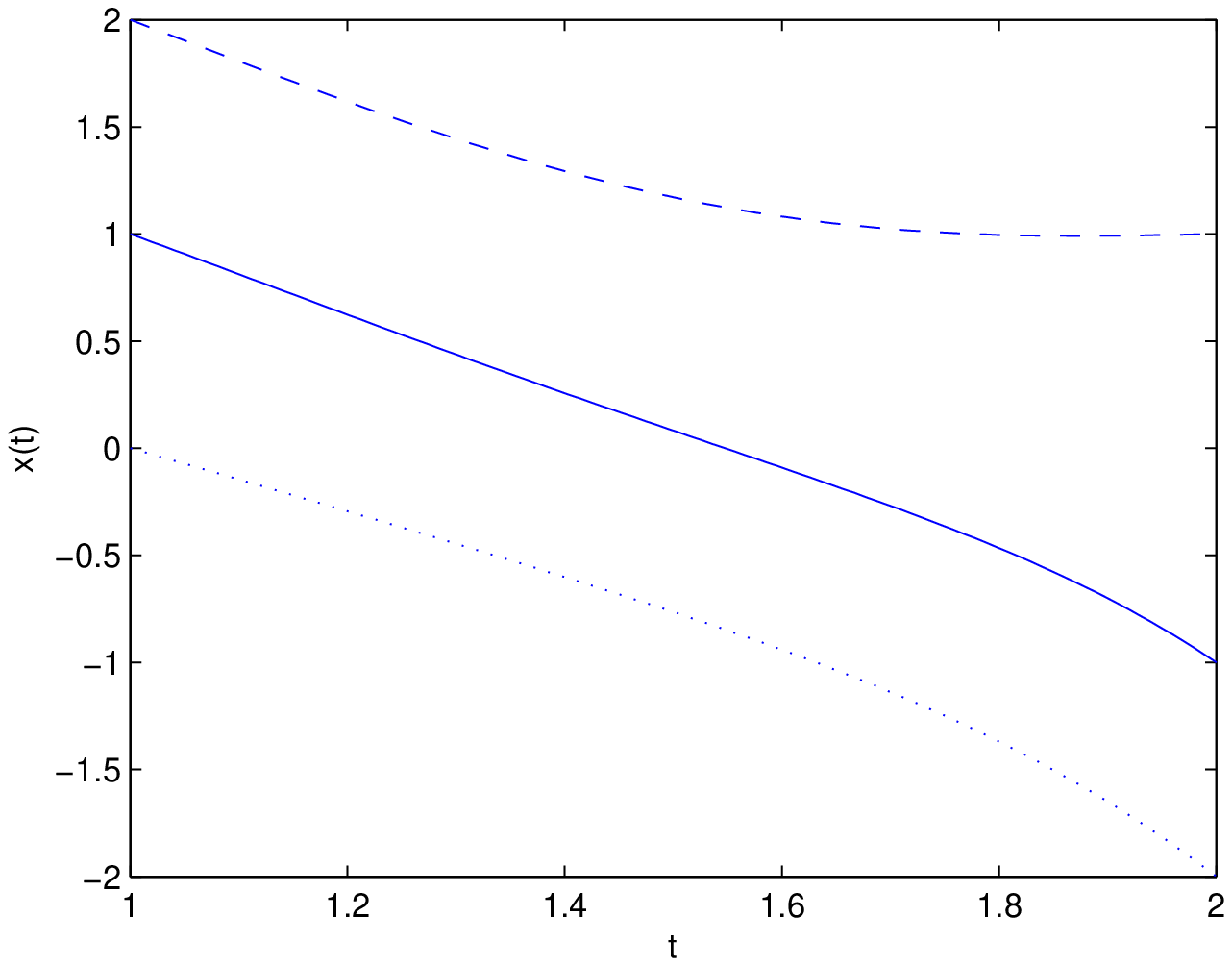}}
\caption{The fuzzy extremals for the fuzzy
optimal control problem \eqref{105} of Example~\ref{example} 
under $[(2)-gH]$-differentiability of $\tilde{x}$.}\label{fig0005}
\end{figure}
Comparing Figures~\ref{fig005} and \ref{fig0005}, we see 
that using $[(2)-gH]$-differentiability of $\tilde{x}$ 
the length of support of $\tilde{u}(t)$ is decreasing.
\end{example}


\subsection{On two examples of Najariyan and Farahi}

In the recent paper \cite{farahi}, Najariyan and Farahi characterize extremals 
for fuzzy linear time-invariant (autonomous) optimal control systems. 
Precisely, they investigate a method for solving the following time-invariant 
fuzzy optimal control problem:
\begin{equation}
\label{d2}
\begin{gathered}
\int_{a}^{b}\tilde{u}^2(t)dt\rightarrow \min,\\
 \mathcal{D}_{gH}\tilde{x}(t)=A\odot \tilde{x}(t)+ C\odot\tilde{u}(t),\\
\tilde{x}(a)=\tilde{x}_{a}, \quad \tilde{x}(b)=\tilde{x}_{b}.
\end{gathered}
\end{equation}
Main result of \cite{farahi} asserts that the fuzzy optimal control problem \eqref{d2}  
is equivalent to the crisp complex optimal control system
\begin{equation}
\label{d3} 
\begin{gathered}
\int_{a}^{b} \left((\underline{u}^r(t))^2
+ i(\overline{u}^r(t))^2\right)dt\rightarrow \min,\\
\dot{\underline{x}^r}(t)+i \dot{\overline{x}^r}(t) =B(\underline{x}^r(t)+i\overline{x}^r(t))+D(\underline{u}^r(t)+i\overline{u}^r(t)),\\
{\underline{x}^r}(a)+i{\overline{x}^r}(a)={\underline{x}^r}_a+i{\overline{x}^r}_a,
\quad {\underline{x}^r}(b)+i{\overline{x}^r}(b)={\underline{x}^r}_b+i{\overline{x}^r}_b, 
\end{gathered}
\end{equation}
where the elements of the matrices $B$ and $D$ are determined 
from those of $A$ and $C$ as follows:
\[
b_{ij}=
\begin{cases}
ea_{ij}  & \text{ if } a_{ij}\geq 0,\\
ga_{ij}  &  \text{ if } a_{ij}< 0,
\end{cases}
\quad
d_{ij}=
\begin{cases}
ec_{ij} &  \text{ if } c_{ij}\geq 0,\\
gc_{ij} & \text{ if }  c_{ij} < 0,
\end{cases}
\]
with $e:a+bi\rightarrow a+bi$ and $g:a+bi\rightarrow b+ai$.
The extremals for the crisp optimal control problem \eqref{d3}  
are given by the classical PMP \cite{Pontryagin}.

\begin{example}[Example 4.2 of Najariyan and Farahi \cite{farahi}]
\label{d5}
Consider the following problem:
\begin{equation}
\label{d07}
\begin{gathered}
\int_{0}^{1} \tilde{u} ^{2}(t)dt \longrightarrow \min, \\
\begin{cases}
\mathcal{D}_{gH}\tilde{x}_1 (t) = -2   \tilde{x}_2(t) +   \tilde{u}(t),\\
\mathcal{D}_{gH}\tilde{x}_2 (t) = 2   \tilde{x}_1(t),
\end{cases}\\
\tilde{x}_1(0)= \tilde{x}_2(0)= (1,2,3),\\
 \tilde{x}_1(1)=  \tilde{x}_2(1)= (-0.5,0,0.5).
\end{gathered}
\end{equation}
In \cite{farahi} the authors provide a figure (see \cite[Figure~2]{farahi}) 
with what they claim to be the fuzzy control and state extremals 
for problem \eqref{d07}. It turns out that the provided functions
are not extremals for the optimal control problem \eqref{d07}.
Indeed, in the crisp case, i.e., when the variables $\tilde{x}_1 (t)$,
$\tilde{x}_2 (t)$ and $\tilde{u}(t)$ and $\tilde{2}=(1,2,3)$ 
and $\tilde{0}=(-0.5,0,0.5)$ are crisp quantities, 
the fuzzy optimal control problem \eqref{d07} 
is transformed into the following crisp optimal control problem:
\begin{equation}
\label{prb:ex02}
\begin{gathered}
\int_{0}^{1}  u^{2}(t) dt \longrightarrow \min, \\
\begin{cases}
\dot{x}_1 (t) = -2 x_2(t) +   u(t),\\
\dot{x}_2 (t) = 2   x_1(t),
\end{cases}\\
x_1(0)= x_2(0)= 2,\\ x_1(1)=  x_2(1)= 0.
\end{gathered}
\end{equation}
The extremals for \eqref{prb:ex02} are easily obtained from 
the classical PMP \cite{Pinch,Pontryagin}. 
Figure~\ref{fig4} shows the graphics of the control and state extremals 
for problem \eqref{prb:ex02}. Comparing these functions with the ones 
given in \cite[Example 4.2]{farahi}, one may conclude that 
there is an inconsistency in \cite[Example 4.2]{farahi}.
\begin{figure} 
\label{001}
\centering
\subfloat[The control extremal $u(t)$]{\label{fig4:a}
\includegraphics[scale=0.42]{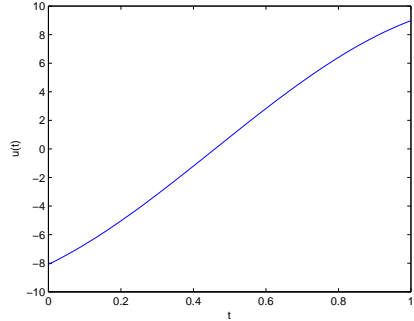}}\\
\subfloat[The state extremal $x_1(t)$]{\label{fig4:b}
\includegraphics[scale=0.42]{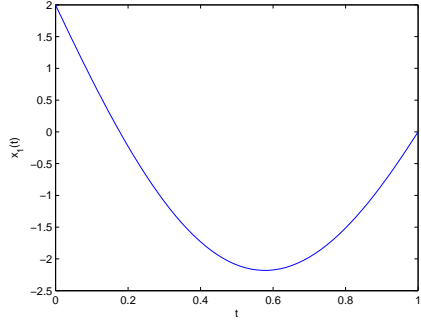}}
\subfloat[The state extremal $x_2(t)$]{\label{fig4:c}
\includegraphics[scale=0.42]{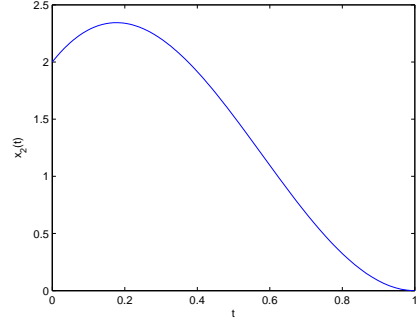}}
\caption{The extremals for the crisp 
optimal control problem \eqref{prb:ex02}
of Example~\ref{d5}.}
\label{fig4}
\end{figure}
Let us use Theorem~\ref{the001} to obtain the  
extremals for \eqref{d07}. Suppose that $\tilde{x}_1$ 
is a [(1)-gH]-differentiable function and  
$\tilde{x}_2$ is a [(2)-gH]-differentiable function. 
The analysis of the other three cases are similar 
and are left to the reader. Our assumption leads to
\begin{equation*}
\mathcal{H}=-((\underline{u}^r)^{2}+(\overline{u}^r)^{2})
+p_1(-2\overline{x}_2^r+\underline{u}^r)
+p_2(-2\underline{x}_2^r+\overline{u}^r)
+p_3(2\overline{x}_1^r)+p_4(2\underline{x}_1^r).
\end{equation*}
From the optimality conditions of Theorem~\ref{the001}, 
the initial conditions and the control system of problem \eqref{d07},
and considering $\beta=1$, we obtain that
\begin{equation}
\label{eq:CO:ex:2}
\begin{cases}
\dot{p}_1(t)=-2{p}_4(t),\\
\dot{p}_2(t)=-2p_3(t),\\
\dot{p}_3(t)=2{p}_2(t),\\
\dot{p}_4(t)=2p_1(t),\\
{\underline{u}}^r(t)= \frac{p_1(t)}{2},\\
{\overline{u}}^r(t)= \frac{p_2(t)}{2},
\end{cases}
\ 
\begin{cases}
\dot{\underline{x}}_1^r(t)= -2\overline{x}_2^r(t)+ \underline{u}^r(t),\\
\dot{\overline{x}}_1^r(t)= -2\underline{x}_2^r(t)+ \overline{u}^r(t),\\
\dot{\underline{x}}_2^r(t)=2\overline{x}_1^r(t),\\
\dot{\overline{x}}_2^r(t)= 2\underline{x}_1^r(t),
\end{cases}
\ 
\begin{cases}
\underline{x}_1^r(0)= \underline{x}_2^r(0)=r+1,\\
\overline{x}_1^r(0)= \overline{x}_2^r(0)=3-r,\\
\underline{x}_1^r(1)= \underline{x}_2^r(1)=-0.5+0.5r,\\
\overline{x}_1^r(1)= \overline{x}_2^r(1)=0.5-0.5r.
\end{cases}
\end{equation}
By solving \eqref{eq:CO:ex:2}, the control and state extremals
can be found straightforwardly. Figure~\ref{fig5} shows the graphics
of the fuzzy control and state extremals, where the continuous lines 
in the center correspond to $r=1$.
\begin{figure}[tbp]
\label{02}
\centering
\subfloat[The fuzzy control extremal $u(t)$]{\label{fig5:a}
\includegraphics[scale=0.42]{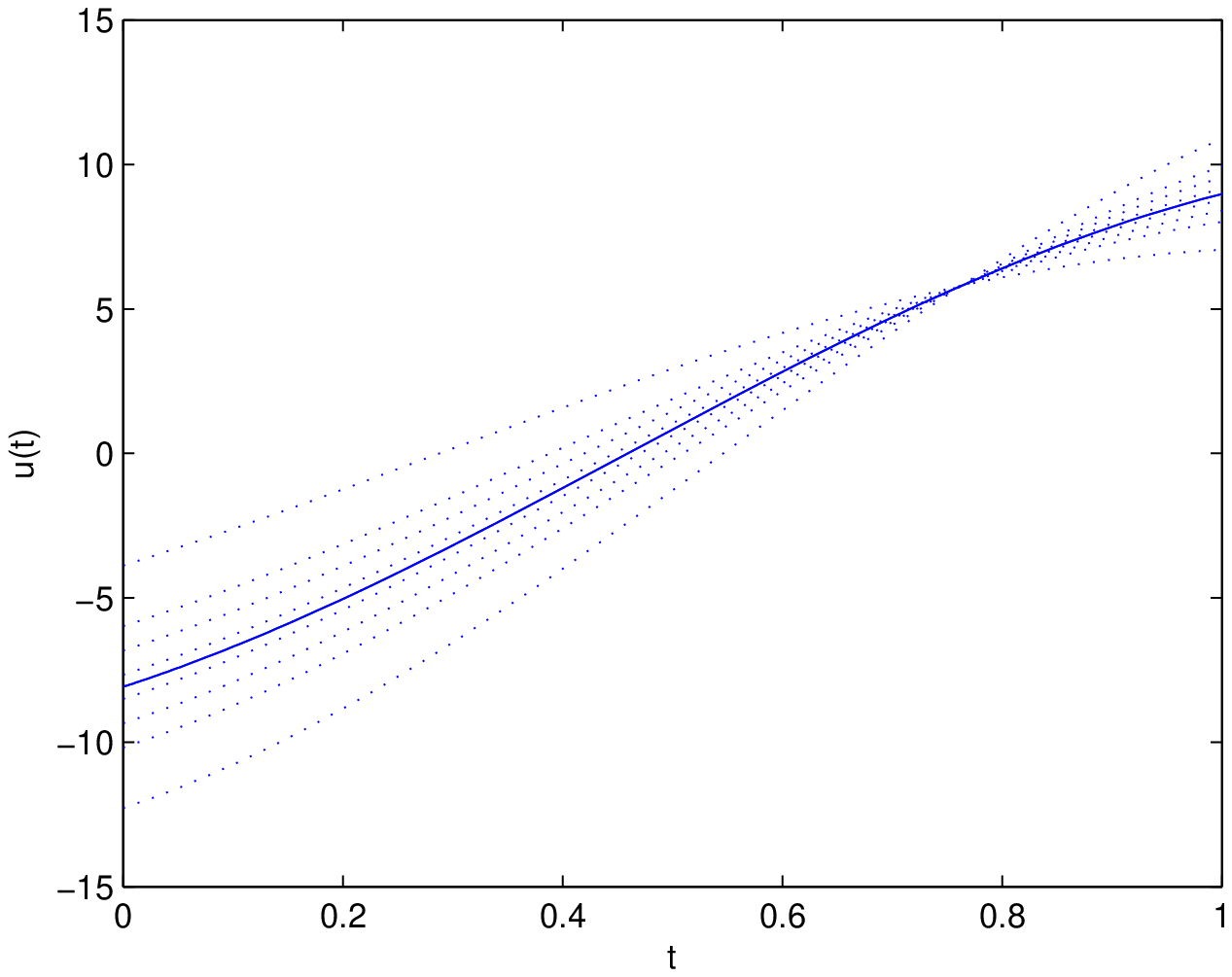}}\\
\subfloat[The fuzzy state extremal $x_1(t)$]{\label{fig5:b}
\includegraphics[scale=0.42]{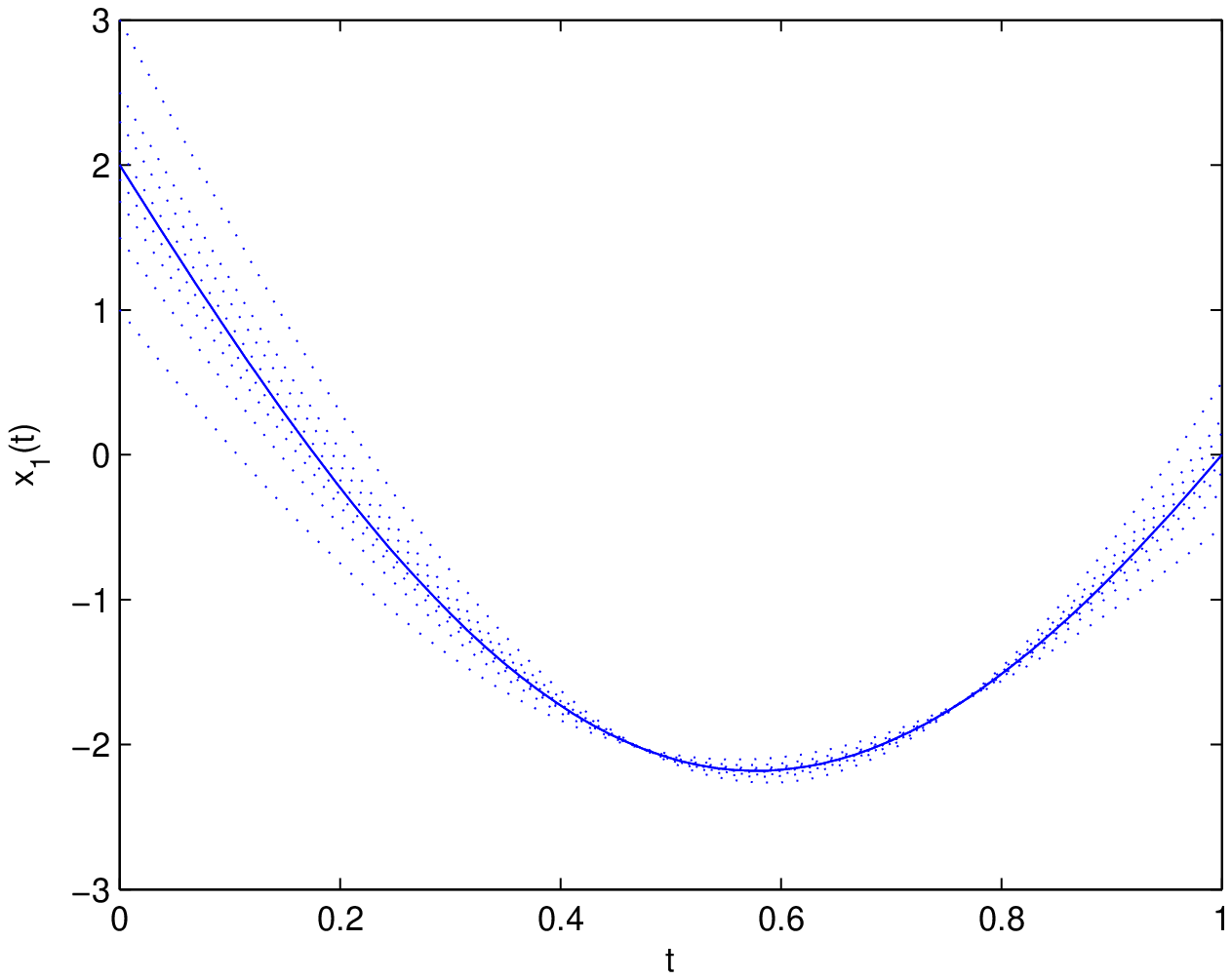}}
\subfloat[The fuzzy state extremal $x_2(t)$]{\label{fig5:c}
\includegraphics[scale=0.42]{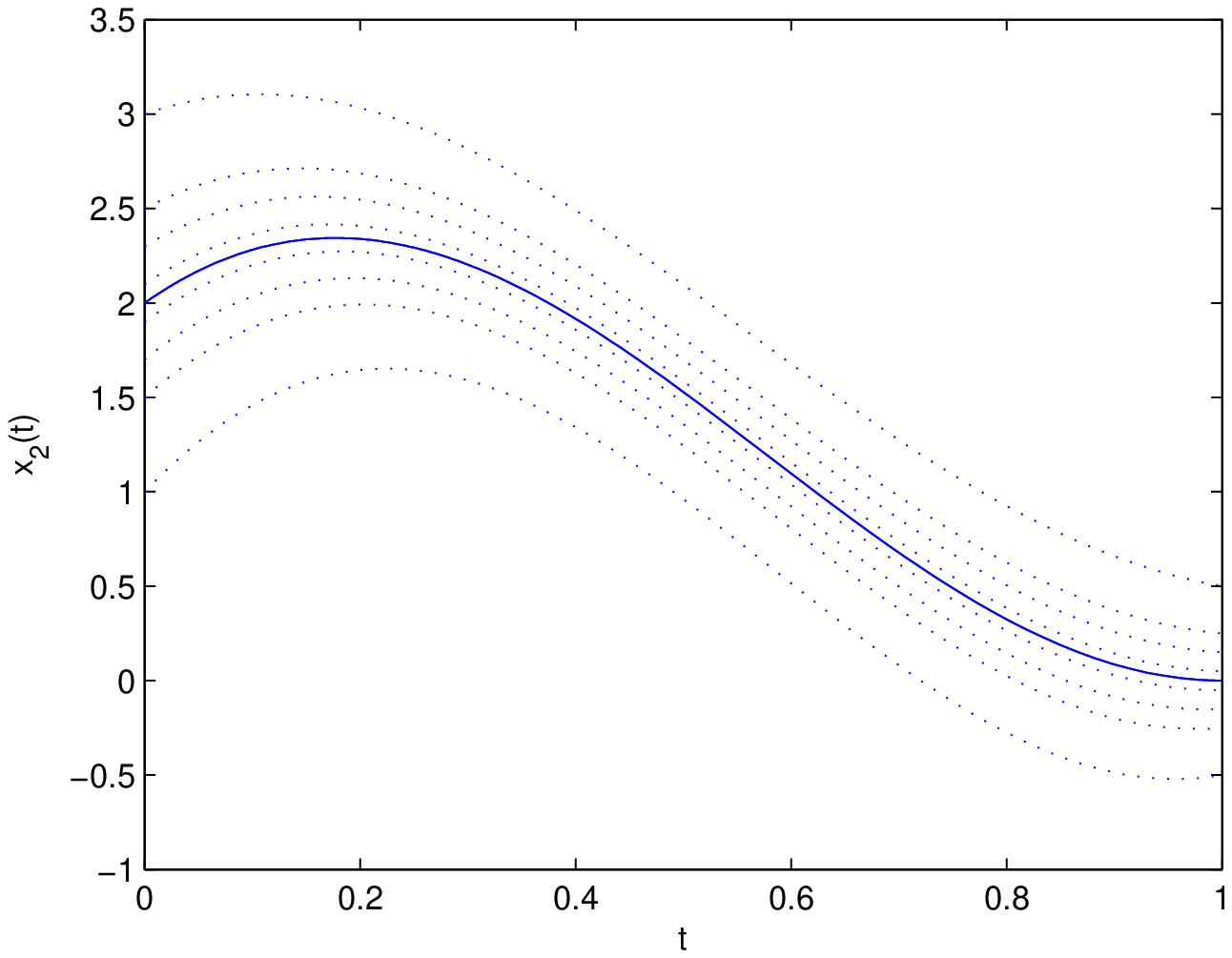}}
\caption{The fuzzy extremals for the fuzzy
optimal control problem \eqref{d07}
of Example~\ref{d5}.}\label{fig5}
\end{figure}
We clearly see from Figures~\ref{fig4} and \ref{fig5}
that the fuzzy extremals of the time-invariant linear  
optimal control problem \eqref{d07} are related with 
the extremals of the crisp optimal control problem 
\eqref{prb:ex02}, which is in agreement 
with the results of \cite{farahi}.
\end{example}

In \cite{farahi2}, Najariyan and Farahi also propose
a method to find extremals for linear non-autonomous 
fuzzy optimal control problems with fuzzy boundary conditions. 
Here we show that fuzzy minimizers 
for \cite[Example~3]{farahi2} do not exist.

\begin{example}[Example~3 of Najariyan and Farahi \cite{farahi2}]
\label{d05}
Consider the following problem:
\begin{equation}
\label{d007}
\begin{gathered}
\int_{0}^{2} \tilde{u} (t)\odot \tilde{u} (t)dt \longrightarrow \min, \\
\mathcal{D}_{gH}\tilde{x} (t) = (2t-1) \odot \tilde{x}(t) + \sin(t) \tilde{u}(t),\\
\tilde{x}(0)= (1,2,3), \quad \tilde{x}(2)=  (-1,0,1).
\end{gathered}
\end{equation}
In \cite[Example~3]{farahi2}, the authors claim to have found the 
control and state extremals for problem \eqref{d007} 
(cf. \cite[Figure~3]{farahi2}). Here we show that in fact the fuzzy state 
and control minimizers do not exist for problem \eqref{d007}. 
Since $diam(\tilde{x}(0))=diam(\tilde{x}(2))=2-2r$, 
we have $\tilde{x}(t)=\tilde{x}(0)+\tilde{f}(t)$, 
where $diam(\tilde{f}(t))=0$. Hence, $D_{gH} \tilde{x}(t)
=D_{gH} \tilde{f}(t)$, that is, $(2t-1) \odot  \tilde{x}(t) 
+ \sin(t) \tilde{u}(t)=D_{gH} \tilde{f}(t)$. 
Consequently, $diam((2t-1) \odot  \tilde{x}(t) 
+ \sin(t)\tilde{u}(t))=0$. Then, 
$\sin t(\overline{u}^r-\underline{u}^r)+(2t-1)(\overline{x}^r
-\underline{x}^r)=0$ for every $t\in[0,2]$ and $r\in[0,1]$. 
Since $\sin(t)>0$ and $2t-1>0$ for $t\in(\frac{1}{2},2]$, we arrive at
$(\overline{u}^r-\underline{u}^r)=(\overline{x}^r-\underline{x}^r)=0$. 
Therefore, $diam(\tilde{x}(t))=diam(\tilde{u}(t))=0$, which 
is impossible, and the fuzzy state and control minimizers do not exist. 
Moreover, note that, in the crisp case, that is, when the variables 
$\tilde{x} (t)$ and $\tilde{u}(t)$ and  
$\tilde{2}=(1,2,3)$ and $\tilde{0}=(-1,0,1)$ are crisp  quantities, 
the fuzzy optimal control problem \eqref{d007} is transformed 
into the following crisp optimal control problem:
\begin{equation}
\label{prb:ex002}
\begin{gathered}
\int_{0}^{2}  u^{2}(t) dt \longrightarrow \min, \\
\dot{x}(t) =(2t-1) x(t) +  \sin(t) u(t),\\
x(0)= 2, \quad x(2)= 0.
\end{gathered}
\end{equation}
The extremals for \eqref{prb:ex002} are easily obtained from 
the classical PMP \cite{Pontryagin}. Figure~\ref{fig2} 
plots the control and state extremals for problem \eqref{prb:ex002}. 
We conclude that the extremals of the crisp optimal control problem 
are not necessarily solution to the original fuzzy  
optimal control problem when $r=1$. This gives new insights
to the results of \cite{farahi2}. 
\begin{figure} 
\label{0001}
\centering
\subfloat[The control extremal]{\label{prb:ex002:a}
\includegraphics[scale=0.42]{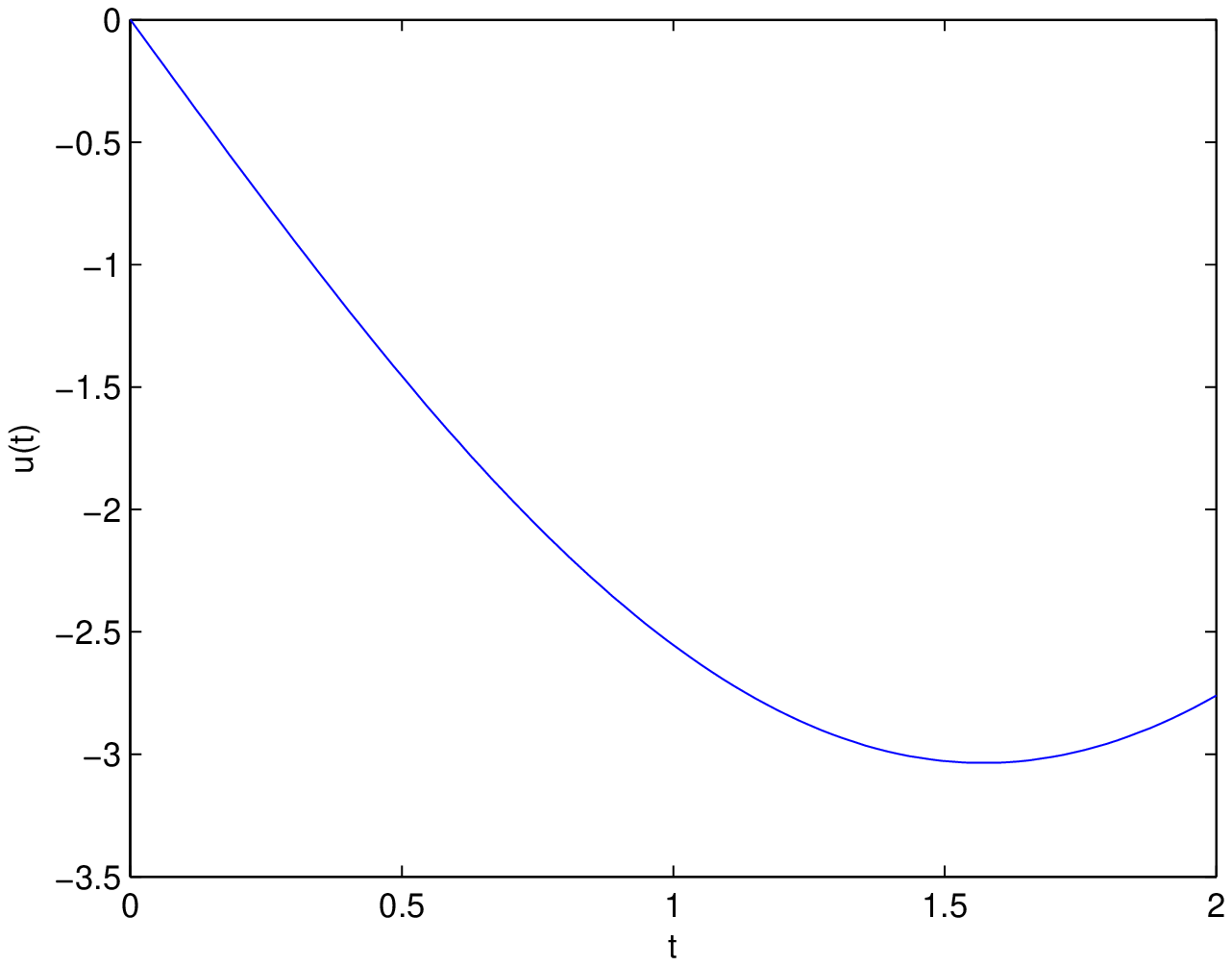}}
\subfloat[The state extremal]{\label{prb:ex002:b}
\includegraphics[scale=0.42]{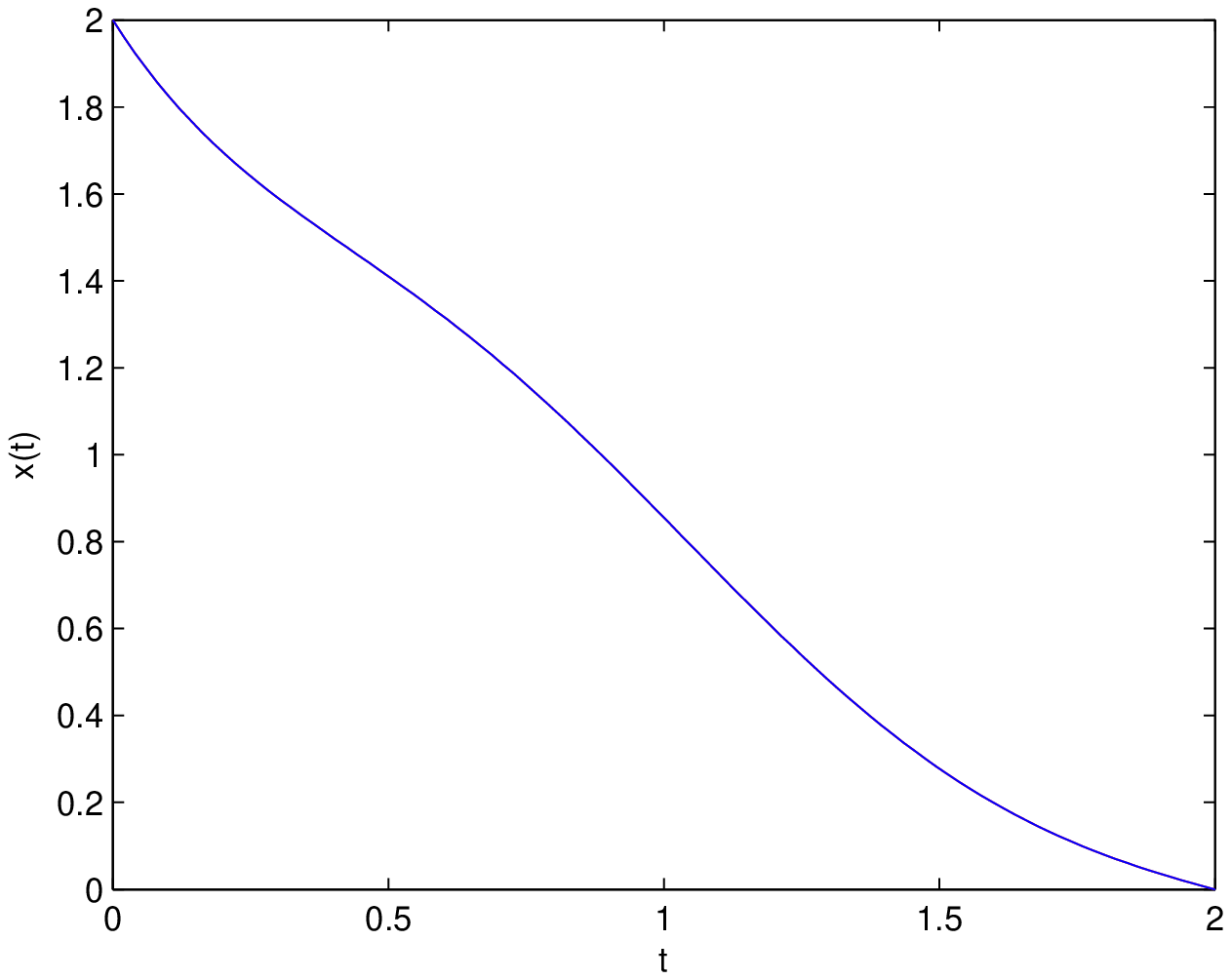}}
\caption{The extremals for the crisp 
optimal control problem \eqref{prb:ex002}
of Example~\ref{d05}.}
\label{fig2}
\end{figure}
\end{example}


\section{Conclusion}
\label{5}

In this paper, a novel technique has been presented 
to solve a class of fuzzy fractional optimal control problems, 
where the coefficients of the system can be time-dependent.
More precisely, we established a weak Pontryagin Maximum 
Principle (PMP) for fuzzy fractional optimal control problems 
depending on generalized Hukuhara fractional Caputo derivatives
(Theorem~\ref{the001}). The results improve those of \cite{farahi,farahi1}, 
where fuzzy optimal controls subject to time-invariant control systems 
are considered. See also Remarks~\ref{remark1} and \ref{remark2}, 
showing that our result easily generalizes those previously obtained 
in \cite{fard2,Farhadinia1}. The main features of our optimality 
conditions were summarized and highlighted with three illustrative examples. 
Two of the examples give interesting insights to the results 
of \cite{farahi,farahi2}.

We have just discussed necessary optimality conditions. 
Much remains to be done and we end by mentioning some possible
lines of research. The obtained fuzzy fractional optimality 
conditions are, in general, difficult to solve and 
it would be good to develop specific numerical methods 
to address the issue. To obtain second order necessary 
optimality conditions is presently a big challenge. 
Other open lines of research consist to prove 
sufficient optimality conditions and existence results.
While here we have assumed that the optimal solution exists, 
and necessary optimality conditions have been obtained 
under such assumption, as Example~\ref{d05} shows,
this is not always the case. As future work, we intend 
to prove conditions assuring the existence of optimal solutions 
to fuzzy fractional optimal control problems. 


\section*{Acknowledgements}

The authors are grateful to Catherine Choquet and to an anonymous
Referee for their helpful and constructive suggestions.



\medskip

Submitted Jun 6, 2016; revised Nov 27, 2016; accepted Feb 3, 2017. 

\medskip


\end{document}